%% file: main.tex
\documentclass{article}
\usepackage[utf8]{inputenc}
\usepackage[margin=1in]{geometry}
\usepackage{amsmath,amsthm,amssymb}
\usepackage{enumerate}
\usepackage{xfrac}
\usepackage[colorlinks=true,hidelinks]{hyperref}
    \AtBeginDocument{}
\usepackage{tikz} 
\usetikzlibrary{arrows}
\usetikzlibrary{patterns}
\usepackage{pgfplots} 
\usetikzlibrary{decorations.markings} 
\usetikzlibrary{3d,calc} 
\usepackage{fancyhdr}
\newcommand{\C}{\mathbb{C}}

\newcommand{\F}{\mathbb{F}}

\newcommand{\N}{\mathbb{N}}
\renewcommand{\P}{\mathbb{P}}
\newcommand{\Q}{\mathbb{Q}}

\newcommand{\Z}{\mathbb{Z}}

\newcommand{\Spec}{\ensuremath{\operatorname{Spec}}}

\newcommand{\proj}{\ensuremath{\operatorname{proj}}}

\newcommand{\orb}{\mathcal{O}}

\newcommand*\cat[1]{{\tt #1}}

\newcommand{\Hom}{\ensuremath{\operatorname{Hom}}}

\newcommand{\Fun}{\ensuremath{\operatorname{Fun}}}

\newcommand{\Gal}{\ensuremath{\operatorname{Gal}}}

\newcommand{\eff}{\ensuremath{\operatorname{eff}}}

\newcommand{\Dec}{\ensuremath{\operatorname{Dec}}}
\newcommand{\Lin}{\ensuremath{\operatorname{Lin}}}

\newcommand{\Int}{\ensuremath{\operatorname{Int}}}

\newcommand{\legen}[2]{\left (\frac{#1}{#2}\right )}

\newcommand{\roof}[3]{\tikz[baseline=10]{
  \node at (-1,0) (a) {$#1$};
  \node at (1,0) (b) {$#2$};
  \node at (0,1) (top) {$#3$};
  \draw[->] (top) -- (a);
  \draw[->] (top) -- (b);
}}
\newcommand{\roofx}[5]{\tikz[baseline=10]{
  \node at (-1,0) (a) {$#1$};
  \node at (1,0) (b) {$#2$};
  \node at (0,1) (top) {$#3$};
  \draw[->] (top) -- (a) node[left,pos=.2] {$#4$};
  \draw[->] (top) -- (b) node[right,pos=.2] {$#5$};
}}

\newcommand{\ds}{\displaystyle}

\newtheorem{thm}{Theorem}[section]
\newtheorem{prop}[thm]{Proposition}

\newtheorem{lem}[thm]{Lemma}
\newtheorem{defn}[thm]{Definition}
\newtheorem{cor}[thm]{Corollary}

\newtheorem{rem}[thm]{Remark}
  \let\oldrem\rem
  \renewcommand{\rem}{\oldrem\normalfont}
\newtheorem{question}[thm]{Question}
\newtheorem{ex}[thm]{Example}
  \let\oldex\ex
  \renewcommand{\ex}{\oldex\normalfont}

\title{Categorifying Quadratic Zeta Functions}
\author{Jon Aycock and Andrew Kobin}
\date{\today}

\begin{document}

\maketitle

\rhead{\thepage}
\cfoot{}

\begin{abstract}
The Dedekind zeta function of a quadratic number field factors as a product of the Riemann zeta function and the $L$-function of a quadratic Dirichlet character. We categorify this formula using objective linear algebra in the abstract incidence algebra of the division poset. 
\end{abstract}

\section{Introduction}
\label{sec:intro}

\input{intro}

\section{Quadratic Number Fields}
\label{sec:quadnf}

\input{quadnf}

\section{Abstract Incidence Algebras}
\label{sec:incalgs}

\input{incalgs}

\section{Proofs of the Main Theorems}
\label{sec:mainthm}

\input{mainthm}

\section{Relative Zeta Functions}
\label{sec:relzeta}

\input{relzeta}

\section{Future Directions}

\input{future}


\end{document}

%% file: intro.tex
Let $K/\Q$ be a number field. Its Dedekind zeta function is defined as the formal expression 
$$
\zeta_{K}(s) = \sum_{\frak{a}\in I_{K}^{+}} \frac{1}{N(\frak{a})^{s}} = \sum_{n = 1}^{\infty} \frac{\#\{\frak{a}\mid N(\frak{a}) = n\}}{n^{s}}
$$
where $s$ is a complex number, $N = N_{K/\Q}$ is the field norm and $I_{K}^{+}$ denotes the set of ideals in the ring of integers $\orb_{K}$ of $K$. For example, $\zeta_{\Q}(s)$ is the Riemann zeta function. The formal Dirichlet series $\zeta_{K}(s)$ converges for complex numbers $s$ with real part greater than $1$ and, like the Riemann zeta function, $\zeta_{K}(s)$ admits a meromorphic continuation to the complex plane, satisfies a functional equation and has a well-formulated Riemann Hypothesis predicting the location of its nontrivial zeroes. Moreover, special values of $\zeta_{K}(s)$ contain rich arithmetic information which remains the subject of ongoing study. 

It is well-known that when $K/\Q$ is a quadratic extension, the zeta function factors in the ring of Dirichlet series as 
\begin{equation}\label{eq:quadzeta}
\zeta_{K}(s) = \zeta_{\Q}(s)L(\chi,s)
\end{equation}
where $L(\chi,s)$ is the Dirichlet $L$-function 
$$
L(\chi,s) = \sum_{n = 1}^{\infty} \frac{\chi(n)}{n^{s}}
$$
associated to the Dirichlet character $\chi = \legen{D}{\cdot}$, $D$ being the discriminant of $K$. In this paper, we lift this formula to an equivalence of linear functors in the abstract incidence algebra of the monoid $\N^{\times}$ using techniques in objective linear algebra: 

\begin{thm}[{Theorem~\ref{thm:mainthmred}}]
\label{thm:introred}
In the incidence algebra $I(\N^{\times})$, there is an equivalence of linear functors 
$$
\widetilde{N}_{*}\zeta_{K} + \zeta_{\Q}*\widetilde{L}(\chi)^{-} \cong \zeta_{\Q}*\widetilde{L}(\chi)^{+}. 
$$
\end{thm}

The symbol $\widetilde{N}_{*}$ is a pushforward map induced by the field norm $N = N_{K/\Q}$. The functors $\widetilde{L}(\chi)^{+}$ and $\widetilde{L}(\chi)^{-}$ in $I(\N^{\times})$ are defined in Section~\ref{sec:mainthm} in such a way that their images in the numerical incidence algebra satisfy $\widetilde{L}(\chi)^{+}(n) - \widetilde{L}(\chi)^{-}(n) = \chi(n)$, thereby recovering equation (\ref{eq:quadzeta}). In this way, Theorem~\ref{thm:introred} may be viewed as an categorical version of equation (\ref{eq:quadzeta}). We also prove a version of Theorem~\ref{thm:introred} in the incidence algebra $I(\N,\mid)$ of the division poset $(\N,\mid)$: 

\begin{thm}[{Theorem~\ref{thm:mainthm}}]
\label{thm:introthm}
In the incidence algebra $I(\N,\mid)$, there is an equivalence of linear functors 
$$
N_{*}\zeta_{K} + \zeta_{\Q}*L(\chi)^{-} \cong \zeta_{\Q}*L(\chi)^{+}. 
$$
\end{thm}

As $(\N,\mid)$ is the decalage of the monoid $\N^{\times}$ (see Example~\ref{ex:divposetdec}), $I(\N^{\times})$ is isomorphic to the reduced incidence subalgebra $\widetilde{I}(\N,\mid)\subseteq I(\N,\mid)$, connecting Theorem~\ref{thm:introthm} to the more familiar reduced incidence algebra of a poset which in this case is isomorphic to the ring of Dirichlet series. We consider Theorem~\ref{thm:introthm} to be the $1$-simplex version of a more general phenomenon, while Theorem~\ref{thm:introred} is the $0$-simplex version. We also relate the two theorems by means of a reduction technique (Theorem~\ref{thm:reduction}). 

The paper is organized as follows. In Section~\ref{sec:quadnf}, we review the basic properties of algebraic number fields needed for discussing Dedekind zeta functions. In Section~\ref{sec:incalgs}, we define {\it decomposition sets}, following \cite{gkt1}, and describe the construction of an incidence algebra for a decomposition set. The proofs of Theorems~\ref{thm:introred} and~\ref{thm:introthm} use the calculus of spans in objective linear algebra, which will be explained in Sections~\ref{sec:OLA} and~\ref{sec:OLAincalgs} before proceeding to the proofs of Theorems~\ref{thm:introred} and~\ref{thm:introthm} in Section~\ref{sec:mainthm}. Section~\ref{sec:reduction} discusses the relation between the theorems, while Section~\ref{sec:genquad} generalizes the theorems to an arbitrary quadratic extension of number fields. Finally, in Section~\ref{sec:relzeta} we introduce a new notion of {\it relative zeta function} and interpret Theorems~\ref{thm:introred} and~\ref{thm:introthm} in this language. 

The authors would like to thank their team members at the first Rethinking Number Theory (RNT) workshop - Karen Acquista, Changho Han and Alicia Lamarche - for their early contributions to this project and for many productive discussions along the way. The authors are grateful to the organizers of RNT - Heidi Goodson, Christelle Vincent and Mckenzie West - for the opportunity to begin this project, as well as to all the RNT participants. Finally, the second author would like to thank David Zureick-Brown for numerous suggestions that improved the article.

%% file: quadnf.tex
Let $K/\Q$ be a number field with ring of integers $\orb_{K}$, norm map $N = N_{K/\Q}$ and Dedekind zeta function $\zeta_{K}(s)$. Then $\zeta_{K}(s)$ has the following product formula: 
$$
\zeta_{K}(s) = \prod_{\frak{p}} \frac{1}{1 - N(\frak{p})^{-s}}. 
$$
Here, the product is over all prime ideals $\frak{p}$. We denote the set of all ideals in $\orb_{K}$ by $I_{K}^{+}$ to distinguish it from $I_{K}$, which is commonly used for the group of fractional ideals. For a fixed prime integer $p\in\Z$ and a fixed prime ideal $\frak{p}$ lying over $p$, let $e(\frak{p}\mid p)$ and $f(\frak{p}\mid p)$ denote the ramification index and inertia degree of $\frak{p}$, respectively. Recall that $p$ has one of the following splitting types: 
\begin{itemize}
    \item {\bf ramified} if $e(\frak{p}\mid p) > 1$ for some $\frak{p}$ lying over $p$;
    \item {\bf split} if $e(\frak{p}\mid p) = f(\frak{p}\mid p) = 1$ for all $\frak{p}$ lying over $p$;
    \item {\bf inert} if $p\orb_{K}$ is itself a prime ideal. 
\end{itemize}
When the splitting types of all primes in the extension $K/\Q$ are known, one can factor $\zeta_{K}(s)$ as a Dirichlet series, that is, into local factors at each prime \emph{integer}. We state this for quadratic extensions below. .

\begin{prop}
\label{prop:quadsplitting}
Let $K/\Q$ be a quadratic number field. Then $\zeta_{K}(s)$ has prime factorization 
$$
\zeta_{K}(s) = \prod_{p} \zeta_{K,p}(s)
$$
where
$$
\zeta_{K,p}(s) = \begin{cases}
    \ds\frac{1}{1 - p^{-s}}, &\text{if $p$ is ramified}\\[1em]
    \ds\frac{1}{(1 - p^{-s})^{2}}, &\text{if $p$ is split}\\[1em]
    \ds\frac{1}{1 - p^{-2s}}, &\text{if $p$ is inert}.
\end{cases}
$$
\end{prop}

We can view each of these local factors as a generating function for ideals over powers of $p$: 
$$
\zeta_{K,p}(s) = \sum_{n = 0}^{\infty} \#\{\frak{a} \mid N(\frak{a}) = p^{n}\}t^{n} = \begin{cases}
    1 + p^{-s} + p^{-2s} + \ldots, &\text{if $p$ is ramified}\\
    1 + 2p^{-s} + 3p^{-2s} + \ldots, &\text{if $p$ is split}\\
    1 + p^{-2s} + p^{-4s} + \ldots, &\text{if $p$ is inert}. 
\end{cases}
$$
This perspective will be explained further in the next section. 

\begin{ex}
A well-known example is when $K = \Q(i)$. In this case, 
$$
p \text{ is } \begin{cases}
    &\text{ramified when } p = 2\\
    &\text{split when } p\equiv 1\!\pmod{4}\\
    &\text{inert when } p\equiv 3\!\pmod{4}. 
\end{cases}
$$
This means the zeta function of $\Q(i)$ can be written 
$$
\zeta_{\Q(i)}(s) = \frac{1}{1 - 2^{-s}}\times\prod_{p\equiv 1\,(4)} \frac{1}{(1 - p^{-s})^{2}}\times\prod_{p\equiv 3\,(4)} \frac{1}{1 - p^{-2s}}. 
$$
\end{ex}

We also fix the following notation for convenience in later proofs. For an integer $a\in\Z$ (resp.~an ideal $\frak{a}\in I_{K}^{+}$) and a prime $p\in\Z$ (resp.~a prime ideal $\frak{p}\in\Spec\orb_{K}$), let $v_{p}(a)$ (resp.~$v_{\frak{p}}(\frak{a})$) denote the power of $p$ (resp.~$\frak{p}$) that divides $a$ (resp.~$\frak{a}$). That is, $a$ can be written $a = \prod_{p} p^{v_{p}(a)}$ where all but finitely many of the $v_{p}(a)$ are $0$ (and a similar expression for $\frak{a}$). 

Finally, when $K/\Q$ is a quadratic extension and $p$ is a prime in $\Z$ that splits in $\orb_{K}$, we fix a choice of prime ideal $\frak{p}$ lying over $p$ and denote by $\overline{\frak{p}}$ its Galois conjugate. The formulas proven in Section~\ref{sec:mainthm} will not depend on this choice, although their proofs will. As we will mention later, this choice can be viewed as a section of $\Spec\orb_{K}\rightarrow\Spec\Z$ over the locus of split primes $p\in\Z$.

%% file: incalgs.tex
\renewcommand{\P}{\mathcal{P}}

In this section, we recall the definition of the incidence algebra for a poset, then give a brief survey of objective linear algebra and the construction of incidence algebras for decomposition sets, following \cite{gkt1} and \cite{kob}. 

\subsection{Incidence Algebras for Posets}
\label{sec:decompsets}

Let $(\P,\leq)$ be a poset. We will write $\P$ for short unless the ordering needs to be specified. For any elements $x,y\in\P$, define the {\it interval} $[x,y]$ to be the subposet 
$$
[x,y] = \{z\in\P\mid x\leq z\leq y\}. 
$$
We say $\P$ is {\it locally finite} if every interval in $\P$ is a finite set. Fix a field $k$. The incidence algebra of $\P$ is the $k$-algebra of functions on the set $\Int(\P)$ of intervals in $\P$, with multiplication defined by a certain convolution product. Following \cite{gkt1} and \cite{kob}, we first define the coalgebra of $\P$ and then take its dual algebra. 

\begin{defn}
The {\bf incidence coalgebra} of a locally finite poset is the free $k$-vector space on the set $\Int(\P)$, 
$$
C(\P) = \bigoplus_{e\in\Int(\P)} ke, 
$$
together with the comultiplication and counit maps defined by 
\begin{align*}
    \Delta : C(\P) &\longrightarrow C(\P)\otimes_{k}C(\P)\\
      [x,y] &\longmapsto \sum_{z\in [x,y]} [x,z]\otimes [z,y]\\
    \delta : C(\P) &\longrightarrow k\\
      [x,y] &\longmapsto \delta_{xy} = \begin{cases}
        1, & x = y\\
        0, & x\not = y
      \end{cases}
\end{align*}
and extended linearly. 
\end{defn}

\begin{defn}
The {\it incidence algebra} of a locally finite poset $\P$ is the dual $k$-vector space $I(\P) = \Hom_{k}(C(\P),k)$, together with the multiplication map defined by 
\begin{align*}
    * : I(\P)\otimes_{k}I(\P) &\longrightarrow I(\P)\\
      f\otimes g &\longmapsto \left (f*g : [x,y] \mapsto \sum_{z\in [x,y]} f([x,z])g([z,y])\right )
\end{align*}
(and extended linearly) and unit $\delta$. 
\end{defn}

For a locally finite poset $\P$, the incidence algebra $I(\P)$ (resp.~the coalgebra $C(\P)$) is always associative and unital (resp.~coassociative and counital); however, it need not be commutative (resp.~cocommutative). For an element $f\in I(\P)$, we will usually write $f(x,y) := f([x,y])$. There is a distinguished element $\zeta\in\P$, called the {\it zeta function}, defined by $\zeta(x,y) = 1$ for all $x\leq y$. The principle of {\it M\"{o}bius inversion} says that $\zeta$ is an invertible element of $I(\P)$ (cf.~\cite[Prop.~3.8]{kob}), whose inverse is called the M\"{o}bius function, defined by the following recursion: 
$$
\zeta^{-1} = \mu : [x,y] \longmapsto \begin{cases}
    1, & x = y\\
    -\sum_{z\in [x,y]} \mu(x,z), & x < y.
\end{cases}
$$

Unsurprisingly, the zeta functions of number theory and algebraic geometry correspond to zeta elements in certain incidence algebras, but to illustrate this connection, it is better to pass to the {\it reduced incidence algebra}. We say two intervals $[x,y],[z,w]\in\Int(\P)$ are isomorphic if they are isomorphic as subposets of $\P$; that is, if there is an order-preserving bijection $\P\rightarrow\P$ taking one interval to the other. 

\begin{defn}
For a locally finite poset $\P$, the {\bf reduced incidence algebra} $\widetilde{I}(\P)$ is the subalgebra of $I(\P)$ consisting of all functions $f : \Int(\P)\rightarrow k$ that are constant on isomorphism classes of intervals in $\P$. 
\end{defn}

Alternatively, the set of isomorphism classes of intervals forms a poset and $\widetilde{I}(\P)$ is the incidence algebra of this poset. By definition, $\zeta$ is always an element of $\widetilde{I}(\P)$. 

\begin{ex}
\label{ex:addposet}
For the poset $(\N_{0},\leq)$ of nonnegative integers with the standard ordering, every interval is isomorphic to one of the form $[0,n]$ for some $n\geq 0$. If $f\in\widetilde{I}(\N_{0},\leq)$, write $f(n) := f(0,n)$. Then $\widetilde{I}(\N_{0},\leq)$ is isomorphic to the algebra of power series: 
\begin{align*}
    \widetilde{I}(\N_{0},\leq) &\longrightarrow k[[t]]\\
    f &\longmapsto \sum_{n = 0}^{\infty} f(n)t^{n}. 
\end{align*}
In this case, the zeta element $\zeta\in\widetilde{I}(\N_{0},\leq)$ corresponds to the geometric power series 
$$
\zeta \longleftrightarrow \sum_{n = 0}^{\infty} t^{n} = \frac{1}{1 - t}. 
$$
\end{ex}

\begin{ex}
\label{ex:divposet}
For the division poset $(\N,\mid)$, every interval is of the form $[1,n]$ for some $n\geq 1$. As above, for each $f\in\widetilde{I}(\N,\mid)$ we can write $f(n) := f(1,n)$. This time, the reduced incidence algebra of $(\N,\mid)$ is isomorphic to the algebra of formal Dirichlet series: 
\begin{align*}
    \widetilde{I}(\N,\mid) &\longrightarrow DS(k)\\
    f &\longmapsto \sum_{n = 1}^{\infty} \frac{f(n)}{n^{s}}. 
\end{align*}
When $k = \C$, the zeta element $\zeta\in\widetilde{I}(\N,\mid)$ corresponds to the Riemann zeta function: 
$$
\zeta \longleftrightarrow \sum_{n = 1}^{\infty} \frac{1}{n^{s}}. 
$$
\end{ex}

\begin{ex}
\label{ex:divposetK}
Let $K/\Q$ be a number field with ring of integers $\orb_{K}$. The set of ideals in $\orb_{K}$ forms a poset under division, written $(I_{K}^{+},\mid)$. When $K = \Q$, this is just $(\N,\mid)$. In general, the field norm induces a poset map $N : (I_{K}^{+},\mid)\rightarrow (\N,\mid)$. This can be used to constructs Dirichlet series for any $f\in\widetilde{I}(I_{K}^{+},\mid)$: 
$$
f \longmapsto F(s) = \sum_{\frak{a}\in I_{K}^{+}} \frac{f(\frak{a})}{N(\frak{a})^{s}}. 
$$
In particular, the zeta element $\zeta\in\widetilde{I}(I_{K}^{+},\mid)$ maps to the Dedekind zeta function of $K$: 
$$
\zeta \longmapsto \sum_{\frak{a}\in I_{K}^{+}} \frac{1}{N(\frak{a})^{s}} = \sum_{n = 1}^{\infty} \frac{\#N^{-1}(n)}{n^{s}}. 
$$
\end{ex}

The operation $f\mapsto F(s)$, which we can think of as ``constructing a generating function for elements of $\widetilde{I}(I_{K}^{+},\mid)$'', is an example of a more general construction of pushforward between incidence algebras (see Example~\ref{ex:pushfwd}). The formalism of objective linear algebra in Section~\ref{sec:OLA} will make this operation easier to handle. Unfortunately, such pushforwards need not preserve convolution, but the images of many $f\in\widetilde{I}(I_{K}^{+},\mid)$ have an interesting factorizations in the ring of Dirichlet series. In fact, formula (\ref{eq:quadzeta}) is precisely such a statement about $\zeta_{K}$ and this will be formalized in Theorem~\ref{thm:introred}. 

Another notion that becomes clearer when stated using objective linear algebra is the passage from full incidence algebra to reduced subalgebra. This is explained further in the next section and again in Section~\ref{sec:reduction}. 

\subsection{Incidence Algebras for Decomposition Sets}
\label{sec:decompsets}

It is possible to construct incidence algebras for a far more general class of objects than posets. For example, Leroux and others \cite{ler,cll,dur} construct incidence algebras for so-called M\"{o}bius categories, which include locally finite posets as a special case via their nerves. Other examples of naturally-occurring incidence algebras include: the Hopf algebra of rooted trees \cite{but,ck}, the Fa\`{a} di Bruno bialgebra \cite{joy,dur} and the Hall algebra of an exact category \cite{dk}. 

All of the above situations have been subsumed by the concept of a {\it decomposition space}, introduced and studied by G\'{a}lvez-Carrillo, Kock and Tonks in \cite{gkt1,gkt2,gkt3,gkt-hla,gkt5}. In this paper, we will only need the notion of a {\it decomposition set}, so in this section we will explain the relevant definitions and properties from \cite{gkt1} in the category of simplicial sets. For a similar discussion of decomposition spaces, see \cite[Sec.~3.2-3.3]{kob}. 

Let $S$ be a simplicial set and fix a field of coefficients $k$. 

\begin{defn}
Let $S : \Delta^{op}\rightarrow\cat{Set}$ be a simplicial set which is locally of finite length (cf.~\cite{gkt2} for a precise definition). The {\bf numerical incidence coalgebra} of $S$ is the free $k$-vector space $C(S)$ on the set $S_{1}$ of $1$-simplices of $S$, together with the comultiplication map 
\begin{align*}
    \Delta : C(S) &\longrightarrow C(S)\otimes C(S)\\
    x &\longmapsto \sum_{\substack{\sigma\in S_{2} \\ d_{1}\sigma = x}} d_{2}\sigma\otimes d_{0}\sigma
\end{align*}
and counit 
\begin{align*}
    \delta : C(S) &\longrightarrow k\\
    x &\longmapsto \begin{cases}
      1, &\text{if $x$ is degenerate}\\
      0, &\text{if $x$ is nondegenerate}.
    \end{cases}
\end{align*}
The {\bf numerical incidence algebra} of $S$ is the dual vector space $I(S) = \Hom(C(S),k)$ together with the multiplication map 
\begin{align*}
    I(S)\otimes I(S) &\longrightarrow I(S)\\
    f\otimes g &\longmapsto \left (f*g : x\mapsto \sum_{\substack{\sigma\in S_{2} \\ d_{1}\sigma = x}} f(d_{2}\sigma)g(d_{0}\sigma)\right )
\end{align*}
and unit $\delta$. 
\end{defn}

To distinguish it from the abstract incidence algebra to be defined below, we will sometimes denote the numerical incidence algebra of $S$ by $I_{\#}(S)$. 

For any simplicial set $S$, the {\it zeta function} of $S$ is the element $\zeta\in I(S)$ defined by $\zeta : x\mapsto 1$ for all $x\in S_{1}$ and extended linearly. In general, $I(S)$ is neither associative or unital as a $k$-algebra. However, when $S$ is a {\it decomposition set}, $I(S)$ is an associative, unital $k$-algebra \cite[Sec.~5.3]{gkt1}. 

The notion of a reduced incidence algebra is not always available for a general decomposition set. However, the authors in \cite{gkt5} describe a general situation in which a reduced incidence algebra exists, namely when $S$ is the {\it decalage} of another decomposition set: $S = \Dec_{\perp}(\widetilde{S})$. By definition, $\Dec_{\perp}(\widetilde{S})$ is a simplicial set whose simplices are shifted by $1$ dimension: $\Dec_{\perp}(\widetilde{S})_{k} = \widetilde{S}_{k + 1}$. For our purposes, this means that when $S = \Dec_{\perp}(\widetilde{S})$, $\widetilde{S}_{1} = S_{0}$, which makes it possible to relate the reduced incidence algebra $\widetilde{I}(S)$ to the full (numerical) incidence algebra $I(\widetilde{S})$. In lieu of a general explanation of the reduction procedure, which is already given in \cite[Sec.~1.5.3]{gkt5}, we focus on three key examples. 

\begin{ex}
Let $S = (\N_{0},\leq)$ be the poset from Example~\ref{ex:addposet} and let $\widetilde{S} = \N_{0}^{+}$ be the additive monoid of nonnegative integers. As both $S$ and $\widetilde{S}$ are categories, hence decomposition sets \cite[Prop.~3.7]{gkt1}, they admit numerical incidence algebras $I(S) = I(\N_{0},\leq)$ and $I(\widetilde{S}) = I(\N_{0}^{+})$. By \cite[Lem.~2.1.2]{gkt5}, $(\N_{0},\mid) \cong \Dec_{\perp}(\N_{0}^{+})$ and this induces an algebra homomorphism 
$$
I(\N_{0}^{+}) \longrightarrow I(\N_{0},\leq)
$$
which is an isomorphism onto the reduced incidence subalgebra $\widetilde{I}(\N_{0},\leq)$. Explicitly, $(\N_{0}^{+})_{1} \cong (\N_{0},\leq)_{0}$ which is further equivalent to the set of isomorphism classes of intervals in $(\N_{0},\leq)$, so $f\in\widetilde{I}(\N_{0},\leq)$ is already a function on the $1$-simplices of $\N_{0}^{+}$; it's easy to check that the algebra structures agree. 
\end{ex}

\begin{ex}
\label{ex:divposetdec}
Let $S = (\N,\mid)$ be the division poset from Example~\ref{ex:divposet} and let $\widetilde{S} = \N^{\times}$ be the monoid of positive integers under multiplication. As above, $S$ and $\widetilde{S}$ admit numerical incidence algebras $I(S) = I(\N,\mid)$ and $I(\widetilde{S}) = I(\N^{\times})$. By \cite[Lem.~2.2.2]{gkt5}, $(\N,\mid) \cong \Dec_{\perp}(\N^{\times})$ and this induces an algebra homomorphism 
$$
I(\N^{\times}) \longrightarrow I(\N,\mid)
$$
which is an isomorphism onto the reduced subalgebra $\widetilde{I}(\N,\mid)$. Moreover, this is consistent with the factorization of $(\N,\mid)$ into prime factors (see \cite[Ex.~3.10]{kob}): 
$$
I(\N^{\times}) \cong \bigotimes_{p} I(\{p^{k}\}^{\times}) \cong \bigotimes_{p} \widetilde{I}(\{p^{k}\},\mid) \cong \widetilde{I}(\N,\mid). 
$$
Explicitly, $(\N^{\times})_{1} \cong (\N,\mid)_{0}$ which is further equivalent to the set of isomorphism classes of intervals in $(\N,\mid)$ and as above, it's easy to check that the algebra structures agree. 
\end{ex}

\begin{ex}
\label{ex:Kdivposetdec}
For a number field $K/\Q$, let $S = (I_{K}^{+},\mid)$ be the ideal division poset from Example~\ref{ex:divposetK} and let $\widetilde{S} = I_{K}^{\times}$ be the monoid of ideals in $\orb_{K}$ under ideal multiplication. A similar proof to that of \cite[Lem.~2.2.2]{gkt5} shows that $(I_{K}^{+},\mid) \cong \Dec_{\perp}(I_{K}^{\times})$ and the induced algebra homomorphism 
$$
I(I_{K}^{\times}) \longrightarrow I(I_{K}^{+},\mid)
$$
is an isomorphism onto the reduced subalgebra $\widetilde{I}(I_{K}^{+},\mid)$. Explicitly, $(I_{K}^{\times})_{1} \cong (I_{K}^{+},\leq)_{0}$ which is further equivalent to the set of isomorphism classes of intervals in $(I_{K}^{+},\mid)$ and the algebra structures agree as before. 
\end{ex}

\subsection{Objective Linear Algebra}
\label{sec:OLA}

It is possible to further categorify the incidence algebras constructed above using Lawvere's {\it objective linear algebra} \cite{law,lm}. Informally, objective linear algebra is ``linear algebra with sets''. Here is a brief summary of how some of the familiar concepts in linear algebra are ``objectified'' in the category of sets. For more details, see \cite{gkt-hla} and \cite[Sec.~3.4]{kob}. 

First, a scalar field $k$ is replaced by the category $\cat{Set}$ of sets. Addition and multiplication of scalars are incarnated as the set-theoretic operations $B + C := B\coprod C$ and $B\times C$ (the cartesian product). Notice that when $B$ and $C$ are finite sets, taking cardinalities recovers ordinary addition and scalar multiplication, but subtraction and inverses don't always lift to $\cat{Set}$. A set $B$ is also thought of as a basis, but we discard the vector space it generates and just remember the basis itself. 

Next, a vector is replaced by a set map $v : X\rightarrow B$. The fibers of $v$ can be thought of as the components of the vector, indexed by $B$, whereby taking cardinality recovers the familiar coordinate description of a vector in the basis $B$. Accordingly, a vector space with basis $B$ is replaced by a slice category $\cat{Set}_{/B}$. Addition is once again disjoint union: $(v : X\rightarrow B) + (w : Y\rightarrow B) := (v\amalg w : X\coprod Y\rightarrow B)$, while scalar multiplication is the composition $Y \times (v : X\rightarrow B) := (Y\times X\xrightarrow{\operatorname{id}\times v} Y\times B\xrightarrow{\proj_{B}} B)$. 

After choosing bases, a matrix is a ``row of column vectors'' (or equivalently, a ``column of row vectors''). This is represented by a span in $\cat{Set}$: 
$$
\roofx{B}{C}{M}{s}{t}. 
$$
Then a linear map with a given matrix is replaced by a {\it linear functor} $\varphi : \cat{Set}_{/B}\rightarrow\cat{Set}_{/C}$, which is a functor of the form $\varphi = t_{!}s^{*}$ for some span as above. In plain terms, a linear map sums the values of $t$ along the fibers of $s$. 

Perhaps illustrating the notion of linear functor more clearly, matrix-vector multiplication is replaced by span composition: for a vector $v : X\rightarrow B$ and a linear functor $\varphi : \cat{Set}_{/B}\rightarrow\cat{Set}_{/C}$, the product $\varphi(v)$ is the composition on the right side in the following diagram: 
$$
\left ( \tikz[xscale=3,yscale=1.7,baseline=30]{
    \node at (1,0) (b) {$B$};
    \node at (2,0) (c) {$C$};
    \node at (.5,.8) (x) {$X$};
    \node at (1.5,.8) (y) {$M$};
    \node at (1,1.6) (z) {$P$};
    \draw[->] (z) -- (x);
    \draw[->] (z) -- (y);
    \draw[->] (x) -- (b) node[right,pos=.4] {$v$};
    \draw[->] (y) -- (b) node[left,pos=.4] {$s$};
    \draw[->] (y) -- (c) node[right,pos=.4] {$t$};
    \draw[->] (z) to[out=-10,in=100] (c);
    \node at (1.8,1.4) {$\varphi(v)$};
}\right )
$$
Here, $P$ is the pullback in the square. 

Many familiar operations on vector spaces have analogues in objective linear algebra, including the following which will be useful to us later. The tensor product of vector spaces is replaced by $\cat{Set}_{/B}\otimes\cat{Set}_{/C} := \cat{Set}_{/B\times C}$. The vector space $\Hom(V,W)$ is replaced by the category $\Lin(B,C) := \Fun^{L}(\cat{Set}_{/B},\cat{Set}_{/C})$ of colimit-preserving linear functors from $\cat{Set}_{/B}$ to $\cat{Set}_{/C}$. This is an objective vector space by \cite[2.10]{gkt-hla}. The dual of a vector space is replaced by the category $(\cat{Set}_{/B})^{*} := \Fun(\cat{Set}_{/B},\cat{Set})$. From the natural equivalence $\Fun(\cat{Set}_{/B},\cat{Set}) \simeq \cat{Set}_{/B}$, we see that $(\cat{Set}_{/B})^{*}$ is again an objective vector space.

\subsection{Incidence Algebras}
\label{sec:OLAincalgs}

With these concepts in hand, it is possible to define the abstract incidence algebra for a decomposition set at the objective level. The construction of an abstract incidence algebra for any decomposition \emph{space} is carried out in \cite{gkt1}, using a generalization of objective linear algebra called {\it homotopy linear algebra}. Since the proofs of Theorems~\ref{thm:introred} and~\ref{thm:introthm} only involve simplicial sets, we only require tools from objective linear algebra. 

\begin{defn}
For a simplicial set $S$, the {\bf incidence coalgebra} of $S$ is the objective vector space $C(S) = \cat{Set}_{/S_{1}}$ equipped with a  comultiplication linear functor $\Delta : C(S)\rightarrow C(S)\otimes C(S)$, represented by the span 
$$
\Delta = \left (\roofx{S_{1}}{S_{1}\times S_{1}}{S_{2}}{d_{1}}{(d_{2},d_{0})}\right ), 
$$
as well as a counit linear functor $\delta : C(S)\rightarrow\cat{Set}$, represented by the span 
$$
\delta = \left (\roofx{S_{1}}{*}{S_{0}}{s_{0}}{}\right ). 
$$
The {\bf incidence algebra} of $S$ is the dual vector space $I(S) := C(S)^{*} = \Fun(\cat{Set}_{/S_{1}},\cat{Set})$, equipped with unit $\delta$ and a multiplication linear functor $m : I(S)\otimes I(S)\rightarrow I(S)$ defined as follows. For $f,g\in I(S)$, their product $m(f,g)$ is the composition $m(f,g) : \cat{Set}_{/S_{1}}\xrightarrow{\Delta}\cat{Set}_{/S_{1}}\otimes\cat{Set}_{/S_{1}}\xrightarrow{f\otimes g}\cat{Set}\otimes\cat{Set}\xrightarrow{\sim}\cat{Set}$. This may also be represented by the outer span in the diagram 
$$
m(f,g) = \left ( \tikz[xscale=3,yscale=1.7,baseline=30]{
        \node at (0,0) (a) {$S_{1}$};
        \node at (1,0) (b) {$S_{1}\times S_{1}$};
        \node at (2,0) (c) {$*$};
        \node at (.5,.8) (x) {$S_{2}$};
        \node at (1.5,.8) (y) {$S_{1}\times S_{1}$};
        \node at (1,1.6) (z) {$P$};
        \draw[->] (z) -- (x);
        \draw[->] (z) -- (y);
        \draw[->] (x) -- (a) node[left,pos=.3] {$d_{1}$};
        \draw[->] (x) -- (b) node[right,pos=.1] {$(d_{2},d_{0})$};
        \draw[->] (y) -- (b) node[right,pos=.8] {$f\times g$};
        \draw[->] (y) -- (c);
    }\right )
$$
where the top square is a pullback square. 
\end{defn}

When $S$ is a decomposition set, \cite[Thm.~7.4]{gkt1} shows that $I(S)$ is an associative, unital monoid object in the category of objective vector spaces and linear functors; that is, $I(S)$ is an objective algebra over $\cat{Set}$. 

\begin{ex}
\label{ex:pushfwd}
One important construction we will need in the next section is the notion of pushforward between incidence algebras. Suppose $S$ and $T$ are decomposition sets and $F : T_{1}\rightarrow S_{1}$ is a set map on their $1$-simplices. This induces a pushforward map between their incidence algebras $F_{*} : I(T)\rightarrow I(S)$ defined on spans as follows: 
$$
\varphi = \left (\roofx{T_{1}}{*}{M}{f}{}\right ) \quad \longmapsto \quad F_{*}\varphi = \left (\roofx{S_{1}}{*}{M}{F\circ f}{}\right ). 
$$
If $F : T_{1}\rightarrow S_{1}$ is of the form $F = G_{1}$ for a simplicial map $G : T\rightarrow S$, we will write $F_{*}$ or $G_{*}$ to mean this pushforward of spans. 
\end{ex}

\renewcommand{\P}{\mathbb{P}}

%% file: mainthm.tex
Fix the following notation. For a quadratic number field $K/\Q$, let $D$ be the discriminant of $K$ and let $\chi = \chi_{D} = \legen{D}{\cdot}$ be the quadratic Dirichlet character mod $D$ attached to $K$. Set $S = (\N,\mid)$ and $T = (I_{K}^{+},\mid)$ and let $N : T\rightarrow S$ be the simplicial map induced by the norm $N_{K/\Q}$. By Example~\ref{ex:pushfwd}, $N$ induces a pushforward map on incidence algebras $N_{*} : I(T)\rightarrow I(S),f\mapsto N_{*}f$. 

Likewise, set $\widetilde{S} = \N^{\times}$ and $\widetilde{T} = I_{K}^{\times}$, so that $I(\widetilde{S})$ is isomorphic to the reduced incidence algebra of $S$ (Example~\ref{ex:divposetdec}) and $I(\widetilde{T})$ is isomorphic to the reduced incidence algebra of $T$ (Example~\ref{ex:Kdivposetdec}). We will also write $\widetilde{N} : \widetilde{T}\rightarrow\widetilde{S}$ and $\widetilde{N}_{*} : I(\widetilde{T})\rightarrow I(\widetilde{S})$ for the norm map and norm pushforward, respectively, on these monoids and their incidence algebras. 

\begin{rem}
It is not true that $\widetilde{N}_{*} : I(\widetilde{T})\rightarrow I(\widetilde{S})$ is the restriction of the norm pushforward on the ``full'' incidence algebra $I(T)$ to the reduced subalgebra. For example, the zeta functor $\zeta_{K}\in I(T)$ lies in the reduced subalgebra $\widetilde{I}(T)\cong I(\widetilde{T})$ but $N_{*}\zeta_{K} \not\cong \widetilde{N}_{*}\zeta_{K}$. This is already easy to verify in the numerical incidence algebra. 
\end{rem}

Formulas involving Dirichlet series and $L$-functions are almost always easier to understand by studying each prime factor individually. In this vein, in order to prove Theorems~\ref{thm:introred} and~\ref{thm:introthm}, we first prove a ``local version'' of each. Unfortunately, the local formulas we establish do not automatically assemble into the global formulas in Theorems~\ref{thm:introred} and~\ref{thm:introthm} (see Remark~\ref{rem:redprod}). However, the local formulas are useful for illustrating the techniques that will be used to prove Theorems~\ref{thm:introred} and~\ref{thm:introthm}. 

The rest of Section~\ref{sec:mainthm} is divided as follows: the local and global formulas in the reduced incidence algebra $\widetilde{I}(\N,\mid)$ are proven in Section~\ref{sec:mainthmred}; the local and global formulas in the full incidence algebra $I(\N,\mid)$ are proven in Section~\ref{sec:mainthmfull}; the relation between the theorems is established in Section~\ref{sec:reduction}; and a generalization to arbitrary quadratic extensions is proven in Section~\ref{sec:genquad}. 

\subsection{The Reduced Version}
\label{sec:mainthmred}

For simplicity of notation, fix a prime $p$ and let $\widetilde{S} = \{p^{k}\}^{\times}$ be the submonoid of $p$th powers in $\N^{\times}$ and let $\widetilde{T} = \{\frak{a} : N(\frak{a}) = p^{k},k\geq 0\}^{\times}$ be the submonoid of $I_{K}^{\times}$ consisting of ideals lying above powers of $p$. Also let $\widetilde{N} : \widetilde{T}\rightarrow \widetilde{S}$ be the restriction of the norm map to these subposets and let $\zeta_{\Q,p}$ and $\zeta_{K,p}$ be the zeta functions in $I(\widetilde{S})$ and $I(\widetilde{T})$, respectively. 

\begin{thm}[{Local Version of Theorem~\ref{thm:introred}}]
\label{thm:mainthmredloc}
For each prime $p$, in the incidence algebra $I(\{p^{k}\}^{\times})$, there is an equivalence of linear functors 
$$
\widetilde{N}_{*}\zeta_{K,p} + \zeta_{\Q,p}*\widetilde{L}_{p}(\chi)^{-} \cong \zeta_{\Q,p}*\widetilde{L}_{p}(\chi)^{+}. 
$$
\end{thm}

\begin{proof}
We first define the terms appearing in the local formula. By Example~\ref{ex:pushfwd}, the first term, $\widetilde{N}_{*}\zeta_{K,p}$, is represented by the span 
$$
\widetilde{N}_{*}\zeta_{K,p} = \left (\tikz[baseline=10]{
  \node at (-1,0) (a) {$\widetilde{S}_{1}$};
  \node at (1,0) (b) {$*$};
  \node at (0,1) (top) {$\widetilde{T}_{1}$};
  \draw[->] (top) -- (a) node[above,pos=.7] {$\widetilde{N}_{1}\;\;\;$};
  \draw[->] (top) -- (b);
}\right ). 
$$
The functors $\widetilde{L}_{p}(\chi)^{+}$ and $\widetilde{L}_{p}(\chi)^{-}$ will be defined by spans 
$$
\widetilde{L}_{p}(\chi)^{+} = \left (\roofx{\widetilde{S}_{1}}{*}{\widetilde{S}_{1}^{+}}{j^{+}}{}\right ) \quad\text{and}\quad \widetilde{L}_{p}(\chi)^{-} = \left (\roofx{\widetilde{S}_{1}}{*}{\widetilde{S}_{1}^{-}}{j^{-}}{}\right )
$$
for ``objective vectors'' $j^{+} : \widetilde{S}_{1}^{+}\rightarrow\widetilde{S}_{1}$ and $j^{-} : \widetilde{S}_{1}^{-}\rightarrow\widetilde{S}_{1}$ to be defined below. The convolutions $\zeta_{\Q,p}*\widetilde{L}_{p}(\chi)^{-}$ and $\zeta_{\Q,p}*\widetilde{L}_{p}(\chi)^{+}$ in the formula will then be defined by the spans
\begin{align*}
    \zeta_{\Q,p}*\widetilde{L}_{p}(\chi)^{-} &= \left ( \tikz[xscale=3,yscale=1.7,baseline=30]{
        \node at (0,0) (a) {$\widetilde{S}_{1}$};
        \node at (1,0) (b) {$\widetilde{S}_{1}\times\widetilde{S}_{1}$};
        \node at (2,0) (c) {$*$};
        \node at (.5,.8) (x) {$\widetilde{S}_{2}$};
        \node at (1.5,.8) (y) {$\widetilde{S}_{1}\times\widetilde{S}_{1}^{-}$};
        \node at (1,1.6) (z) {$\widetilde{B}^{-}$};
        \draw[->] (z) -- (x) node[left,pos=.3] {$\alpha^{-}$};
        \draw[->] (z) -- (y);
        \draw[->] (x) -- (a) node[left,pos=.3] {$d_{1}$};
        \draw[->] (x) -- (b) node[right,pos=.1] {$(d_{2},d_{0})$};
        \draw[->] (y) -- (b) node[right,pos=.8] {$id\times j^{-}$};
        \draw[->] (y) -- (c);
    }\right )\\
    \text{and}\quad \zeta_{\Q,p}*\widetilde{L}_{p}(\chi)^{+} &= \left ( \tikz[xscale=3,yscale=1.7,baseline=30]{
        \node at (0,0) (a) {$\widetilde{S}_{1}$};
        \node at (1,0) (b) {$\widetilde{S}_{1}\times\widetilde{S}_{1}$};
        \node at (2,0) (c) {$*$};
        \node at (.5,.8) (x) {$\widetilde{S}_{2}$};
        \node at (1.5,.8) (y) {$\widetilde{S}_{1}\times\widetilde{S}_{1}^{+}$};
        \node at (1,1.6) (z) {$\widetilde{B}^{+}$};
        \draw[->] (z) -- (x) node[left,pos=.3] {$\alpha^{+}$};
        \draw[->] (z) -- (y);
        \draw[->] (x) -- (a) node[left,pos=.3] {$d_{1}$};
        \draw[->] (x) -- (b) node[right,pos=.1] {$(d_{2},d_{0})$};
        \draw[->] (y) -- (b) node[right,pos=.8] {$id\times j^{+}$};
        \draw[->] (y) -- (c);
    }\right )
\end{align*}
We now divide the proof into three cases corresponding to the splitting type of $p$ in $\orb_{K}$. 

First, suppose $p$ ramifies in $\orb_{K}$. Let $\widetilde{S}_{1}^{+} = *$ and $j^{+} : \widetilde{S}_{1}^{+}\hookrightarrow\widetilde{S}_{1}$ be the inclusion of the point as the neutral element $1 = p^{0}$. Also, let $\widetilde{S}_{1}^{-} = \varnothing$ with $j^{-} : \widetilde{S}_{1}^{-}\rightarrow\widetilde{S}_{1}$ the unique map. Then $\widetilde{L}_{p}(\chi)^{-}$ and $\zeta_{\Q,p}*\widetilde{L}_{p}(\chi)^{-}$ are both the empty functor, so the local formula in the ramified case reduces to 
\begin{equation}\label{eq:localramred}
\widetilde{N}_{*}\zeta_{K,p} \cong \zeta_{\Q,p}*\widetilde{L}_{p}(\chi)^{+}. 
\end{equation}
For convenience, we abbreviate the span representing the right side of this formula as 
$$
\zeta_{\Q,p}*\widetilde{L}_{p}(\chi)^{+} = \left (\roofx{\widetilde{S}_{1}}{*}{\widetilde{B}^{+}}{d_{1}\circ\alpha^{+}}{}\right ). 
$$
To prove local formula (\ref{eq:localramred}), i.e.~to show the linear functors are equivalent, we will construct an isomorphism $\beta : \widetilde{T}_{1}\rightarrow \widetilde{B}^{+}$ that makes the following diagram commute: 
\begin{center}
\begin{tikzpicture}[scale=1.3]
  \node at (-1,0) (a) {$\widetilde{S}_{1}$};
  \node at (0,1) (b) {$\widetilde{B}^{+}$};
  \node at (0,-1) (c) {$\widetilde{T}_{1}$};
  \node at (1,0) (d) {$*$};
  \draw[->] (b) -- (a) node[left,pos=.3] {$d_{1}\circ\alpha^{+}$};
  \draw[->] (b) -- (d);
  \draw[->] (c) -- (a) node[left,pos=.4] {$\widetilde{N}_{1}$};
  \draw[->] (c) -- (d);
  \draw[->] (c) -- (b) node[right,pos=.5] {$\beta$};
\end{tikzpicture}
\end{center}
By construction, $\widetilde{B}^{+}$ is the following subset of $\widetilde{S}_{2}$: 
$$
\widetilde{B}^{+} = \{\sigma\in\widetilde{S}_{2}\mid d_{0}\sigma = p^{0}\}
$$
and $d_{1}\circ\alpha^{+}$ is just the restriction of $d_{1}$ to this subset. Then we define $\beta : \widetilde{T}_{1}\rightarrow \widetilde{B}^{+}$ by 
$$
\beta(\frak{p}^{k}) = \tikz[baseline=5]{
    \draw[thick] (-.7,0) -- (0,1) node[left,pos=.6] {$p^{k}$};
    \draw[thick] (0,1) -- (.7,0) node[right,pos=.4] {$p^{0}$};
    \draw[thick] (.7,0) -- (-.7,0) node[below,pos=.5] {$p^{k}$};
    \node at (0,.4) {\large $\sigma$};
}
$$
This is clearly invertible since every $2$-simplex in $\widetilde{B}^{+}$ is of this form. Moreover, $\beta$ makes the diagram commute since $N(\frak{p}^{k}) = p^{k}$ when $p$ is ramified. This shows formula (\ref{eq:localramred}) holds in the ramified case. 

Next, suppose $p$ splits in $\orb_{K}$. Let $\widetilde{S}_{1}^{-} = \varnothing$ and $j^{-} : \widetilde{S}_{1}^{-}\rightarrow\widetilde{S}_{1}$ be the same as above and let $\widetilde{S}_{1}^{+} = \widetilde{S}_{1}$ with $j^{+} : \widetilde{S}_{1}^{+}\rightarrow\widetilde{S}_{1}$ the identity map. As in the ramified case, the local formula reduces to local formula (\ref{eq:localramred}) and we can represent $\zeta_{\Q,p}*\widetilde{L}_{p}(\chi)^{+}$ by 
$$
\zeta_{\Q,p}*\widetilde{L}_{p}(\chi)^{+} = \left (\roofx{\widetilde{S}_{1}}{*}{\widetilde{B}^{+}}{d_{1}\circ\alpha^{+}}{}\right ). 
$$
Here, $\widetilde{B}^{+} = \widetilde{S}_{2}$, $d_{1}\circ\alpha^{+} = d_{1}$ and the map $\beta : \widetilde{T}_{1}\rightarrow \widetilde{B}^{+}$ is defined by 
$$
\beta(\frak{p}^{k}\overline{\frak{p}}^{\ell}) = \tikz[baseline=5]{
    \draw[thick] (-.7,0) -- (0,1) node[left,pos=.6] {$p^{k}$};
    \draw[thick] (0,1) -- (.7,0) node[right,pos=.4] {$p^{\ell}$};
    \draw[thick] (.7,0) -- (-.7,0) node[below,pos=.5] {$p^{k + \ell}$};
    \node at (0,.4) {\large $\sigma$};
}
$$
This is an isomorphism since every $\sigma\in\widetilde{S}_{2}$ is of the form shown. Moreover, $\beta$ makes the relevant diagram commute since 
$$
d_{1}\circ\beta(\frak{p}^{k}\overline{\frak{p}}^{\ell}) = d_{1}\left (\tikz[baseline=5]{
    \draw[thick] (-.7,0) -- (0,1) node[left,pos=.6] {$p^{k}$};
    \draw[thick] (0,1) -- (.7,0) node[right,pos=.4] {$p^{\ell}$};
    \draw[thick] (.7,0) -- (-.7,0) node[below,pos=.5] {$p^{k + \ell}$};
    \node at (0,.4) {\large $\sigma$};
}\right ) = p^{k + \ell} = N(\frak{p}^{k}\overline{\frak{p}}^{\ell}). 
$$
This proves the local formula in the split case. 

Finally, assume $p$ is inert and define $\widetilde{S}_{1}^{+}$ and $\widetilde{S}_{1}^{-}$ to be the subsets of even and odd powers of $p$ in $\widetilde{S}_{1} = \{p^{k}\}$, respectively: 
$$
\widetilde{S}_{1}^{+} = \{p^{2k}\} \qquad\text{and}\qquad \widetilde{S}_{1}^{-} = \{p^{2k + 1}\}. 
$$
Also let $j^{+} : \widetilde{S}_{1}^{+}\hookrightarrow\widetilde{S}_{1}$ and $j^{-} : \widetilde{S}_{1}^{-}\hookrightarrow\widetilde{S}_{1}$ be the natural inclusions. Then $\zeta_{\Q,p}*\widetilde{L}_{p}(\chi)^{+}$ and $\zeta_{\Q,p}*\widetilde{L}_{p}(\chi)^{-}$ are represented by spans that can be abbreviated as 
$$
\zeta_{\Q,p}*\widetilde{L}_{p}(\chi)^{+} = \left (\roofx{\widetilde{S}_{1}}{*}{\widetilde{B}^{+}}{d_{1}\circ\alpha^{+}}{}\right ) \quad\text{and}\quad \zeta_{\Q,p}*\widetilde{L}_{p}(\chi)^{-} = \left (\roofx{\widetilde{S}_{1}}{*}{\widetilde{B}^{-}}{d_{1}\circ\alpha^{-}}{}\right ). 
$$
This time, $\widetilde{B}^{+}$ and $\widetilde{B}^{-}$ are the following subsets of $\widetilde{S}_{2}$: 
$$
\widetilde{B}^{+} = \{\sigma\in\widetilde{S}_{2}\mid d_{0}\sigma = p^{2k},k\geq 0\} \quad\text{and}\quad \widetilde{B}^{-} = \{\sigma\in\widetilde{S}_{2}\mid d_{0}\sigma = p^{2k + 1},k\geq 0\}. 
$$
The maps $d_{1}\circ\alpha^{+}$ and $d_{1}\circ\alpha^{-}$ are each just the restriction of $d_{1}$ to the above subsets. To prove the local formula in this case, we must construct an isomorphism $\beta : \widetilde{T}_{1}\coprod \widetilde{B}^{-}\rightarrow \widetilde{B}^{+}$ making the following diagram commute: 
\begin{center}
\begin{tikzpicture}[scale=1.3]
  \node at (-1,0) (a) {$\widetilde{S}_{1}$};
  \node at (0,1) (b) {$\widetilde{B}^{+}$};
  \node at (0,-1) (c) {$\widetilde{T}_{1}\coprod \widetilde{B}^{-}$};
  \node at (1,0) (d) {$*$};
  \draw[->] (b) -- (a) node[left,pos=.3] {$d_{1}$};
  \draw[->] (b) -- (d);
  \draw[->] (c) -- (a) node[left,pos=.4] {$\widetilde{N}_{1}\sqcup d_{1}$};
  \draw[->] (c) -- (d);
  \draw[->] (c) -- (b) node[right,pos=.5] {$\beta$};
\end{tikzpicture}
\end{center}
For $p^{k}\in\widetilde{T}_{1}$, set 
$$
\beta(p^{k}) = \tikz[baseline=5]{
    \draw[thick] (-.7,0) -- (0,1) node[left,pos=.6] {$p^{0}$};
    \draw[thick] (0,1) -- (.7,0) node[right,pos=.4] {$p^{2k}$};
    \draw[thick] (.7,0) -- (-.7,0) node[below,pos=.5] {$p^{2k}$};
    \node at (0,.4) {\large $\sigma$};
}
$$
On the other hand, define $\beta$ on $\widetilde{B}^{-}$ by
$$
\beta\left (\tikz[baseline=5]{
    \draw[thick] (-.7,0) -- (0,1) node[left,pos=.6] {$p^{k}$};
    \draw[thick] (0,1) -- (.7,0) node[right,pos=.4] {$p^{2\ell + 1}$};
    \draw[thick] (.7,0) -- (-.7,0) node[below,pos=.5] {$p^{k + 2\ell + 1}$};
    \node at (0,.4) {\large $\sigma$};
}\right ) = \tikz[baseline=5]{
    \draw[thick] (-.7,0) -- (0,1) node[left,pos=.6] {$p^{k + 1}$};
    \draw[thick] (0,1) -- (.7,0) node[right,pos=.4] {$p^{2\ell}$};
    \draw[thick] (.7,0) -- (-.7,0) node[below,pos=.5] {$p^{k + 2\ell + 1}$};
    \node at (0,.4) {\large $\tau$};
}
$$
This defines $\beta$ on all of $\widetilde{T}_{1}\coprod \widetilde{B}^{-}$. To see that it's a bijection, take $\tau\in \widetilde{B}^{+}$. If $d_{2}\tau = p^{0}$, then $d_{0}\tau = p^{2k}$ for some $k\geq 0$; set $\beta^{-1}(\tau) = p^{k}\in\widetilde{T}_{1}$. Otherwise, $\tau$ is of the form shown below and we can define $\beta^{-1}(\tau)$ to be 
$$
\beta^{-1}\left (\tikz[baseline=5]{
    \draw[thick] (-.7,0) -- (0,1) node[left,pos=.6] {$p^{k + 1}$};
    \draw[thick] (0,1) -- (.7,0) node[right,pos=.4] {$p^{2\ell}$};
    \draw[thick] (.7,0) -- (-.7,0) node[below,pos=.5] {$p^{k + 2\ell + 1}$};
    \node at (0,.4) {\large $\tau$};
}\right ) = \tikz[baseline=5]{
    \draw[thick] (-.7,0) -- (0,1) node[left,pos=.6] {$p^{k}$};
    \draw[thick] (0,1) -- (.7,0) node[right,pos=.4] {$p^{2\ell + 1}$};
    \draw[thick] (.7,0) -- (-.7,0) node[below,pos=.5] {$p^{k + 2\ell + 1}$};
    \node at (0,.4) {\large $\sigma$};
}
$$
This is visibly an inverse to $\beta$, so we have our isomorphism. Finally, $\beta$ makes the relevant diagram commute because: 
\begin{align*}
    & d_{1}\circ\beta(p^{k}) = d_{1}\left (\tikz[baseline=5]{
    \draw[thick] (-.7,0) -- (0,1) node[left,pos=.6] {$p^{0}$};
    \draw[thick] (0,1) -- (.7,0) node[right,pos=.4] {$p^{2k}$};
    \draw[thick] (.7,0) -- (-.7,0) node[below,pos=.5] {$p^{2k}$};
    \node at (0,.4) {\large $\sigma$};
}\right ) = p^{2k} = \widetilde{N}(p^{k})\\
    \text{and}\quad & d_{1}\circ\beta\left (\tikz[baseline=5]{
    \draw[thick] (-.7,0) -- (0,1) node[left,pos=.6] {$p^{k}$};
    \draw[thick] (0,1) -- (.7,0) node[right,pos=.4] {$p^{2\ell + 1}$};
    \draw[thick] (.7,0) -- (-.7,0) node[below,pos=.5] {$p^{k + 2\ell + 1}$};
    \node at (0,.4) {\large $\sigma$};
}\right ) = d_{1}\left (\tikz[baseline=5]{
    \draw[thick] (-.7,0) -- (0,1) node[left,pos=.6] {$p^{k + 1}$};
    \draw[thick] (0,1) -- (.7,0) node[right,pos=.4] {$p^{2\ell}$};
    \draw[thick] (.7,0) -- (-.7,0) node[below,pos=.5] {$p^{k + 2\ell + 1}$};
    \node at (0,.4) {\large $\tau$};
}\right ) = p^{k + 2\ell + 1} = d_{1}\left (\tikz[baseline=5]{
    \draw[thick] (-.7,0) -- (0,1) node[left,pos=.6] {$p^{k}$};
    \draw[thick] (0,1) -- (.7,0) node[right,pos=.4] {$p^{2\ell + 1}$};
    \draw[thick] (.7,0) -- (-.7,0) node[below,pos=.5] {$p^{k + 2\ell + 1}$};
    \node at (0,.4) {\large $\sigma$};
}\right ). 
\end{align*}
This completes the proof of the local formula in all cases. 
\end{proof}

\begin{rem}
\label{rem:splitnotcanonical}
In the split case, the isomorphism $\widetilde{N}_{*}\zeta_{K,p} \cong \zeta_{\Q,p}*\widetilde{L}_{p}(\chi)^{+}$ is \emph{not canonical}: it depends on the choice of $\frak{p}$ and $\overline{\frak{p}}$ lying over $p$. Geometrically, this may be viewed as a choice of section of $\Spec\orb_{K}\rightarrow\Spec\Z$ over the split locus. 
\end{rem}

\begin{rem}
\label{rem:redprod}
We would like to assemble the local formulas of Theorem~\ref{thm:mainthmredloc} into the global formula of Theorem~\ref{thm:introred}. Unfortunately, assembling the formulas in the most straightforward way is not possible since 
$$
\prod_{p} (\widetilde{N}_{*}\zeta_{K,p} + \zeta_{\Q,p}*\widetilde{L}_{p}(\chi)^{-}) \,\not\cong\, \widetilde{N}_{*}\zeta_{K} + \zeta_{\Q}*\prod_{p} \widetilde{L}_{p}(\chi)^{-}. 
$$
However, the techniques in the proof above adapt well to the global setting if we are careful about how each term is represented objectively.
\end{rem}

Next, we prove the global formula in the reduced incidence algebra. Let $\widetilde{S}$ and $\widetilde{T}$ now denote the full monoids $\N^{\times}$ and $I_{K}^{\times}$, respectively. 

\begin{thm}[{Theorem~\ref{thm:introred}}]
\label{thm:mainthmred}
In the reduced incidence algebra $\widetilde{I}(S) = \widetilde{I}(\N,\mid) \cong I(\N^{\times})$, there is an equivalence of linear functors 
$$
\widetilde{N}_{*}\zeta_{K} + \zeta_{\Q}*\widetilde{L}(\chi)^{-} \cong \zeta_{\Q}*\widetilde{L}(\chi)^{+}. 
$$
\end{thm}

\begin{proof}
As in the local proof, the terms in the global formula will be represented by spans 
\begin{align*}
    \widetilde{N}_{*}\zeta_{K} &= \left (\tikz[baseline=10]{
      \node at (-1,0) (a) {$\widetilde{S}_{1}$};
      \node at (1,0) (b) {$*$};
      \node at (0,1) (top) {$\widetilde{T}_{1}$};
      \draw[->] (top) -- (a) node[above,pos=.7] {$\widetilde{N}_{1}\;\;\;$};
      \draw[->] (top) -- (b);
    }\right )\\
    \zeta_{\Q}*\widetilde{L}(\chi)^{-} &= \left ( \tikz[xscale=3,yscale=1.7,baseline=30]{
        \node at (0,0) (a) {$\widetilde{S}_{1}$};
        \node at (1,0) (b) {$\widetilde{S}_{1}\times\widetilde{S}_{1}$};
        \node at (2,0) (c) {$*$};
        \node at (.5,.8) (x) {$\widetilde{S}_{2}$};
        \node at (1.5,.8) (y) {$\widetilde{S}_{1}\times\widetilde{S}_{1}^{-}$};
        \node at (1,1.6) (z) {$\widetilde{B}^{-}$};
        \draw[->] (z) -- (x) node[left,pos=.3] {$\alpha^{-}$};
        \draw[->] (z) -- (y);
        \draw[->] (x) -- (a) node[left,pos=.3] {$d_{1}$};
        \draw[->] (x) -- (b) node[right,pos=.1] {$(d_{2},d_{0})$};
        \draw[->] (y) -- (b) node[right,pos=.8] {$id\times j^{-}$};
        \draw[->] (y) -- (c);
    }\right )\\
    \text{and}\quad \zeta_{\Q}*\widetilde{L}(\chi)^{+} &= \left ( \tikz[xscale=3,yscale=1.7,baseline=30]{
        \node at (0,0) (a) {$\widetilde{S}_{1}$};
        \node at (1,0) (b) {$\widetilde{S}_{1}\times\widetilde{S}_{1}$};
        \node at (2,0) (c) {$*$};
        \node at (.5,.8) (x) {$\widetilde{S}_{2}$};
        \node at (1.5,.8) (y) {$\widetilde{S}_{1}\times\widetilde{S}_{1}^{+}$};
        \node at (1,1.6) (z) {$\widetilde{B}^{+}$};
        \draw[->] (z) -- (x) node[left,pos=.3] {$\alpha^{+}$};
        \draw[->] (z) -- (y);
        \draw[->] (x) -- (a) node[left,pos=.3] {$d_{1}$};
        \draw[->] (x) -- (b) node[right,pos=.1] {$(d_{2},d_{0})$};
        \draw[->] (y) -- (b) node[right,pos=.8] {$id\times j^{+}$};
        \draw[->] (y) -- (c);
    }\right )
\end{align*}
We will define $j^{+} : \widetilde{S}_{1}^{+}\rightarrow\widetilde{S}_{1}$ and $j^{-} : \widetilde{S}_{1}^{-}\rightarrow\widetilde{S}_{1}$. Then the bulk of the proof involves constructing an isomorphism $\beta : \widetilde{T}_{1}\coprod\widetilde{B}^{-}\rightarrow\widetilde{B}^{+}$ making the following diagram commute: 
\begin{center}
\begin{tikzpicture}[scale=1.3]
  \node at (-1,0) (a) {$\widetilde{S}_{1}$};
  \node at (0,1) (b) {$\widetilde{B}^{+}$};
  \node at (0,-1) (c) {$\widetilde{T}_{1}$};
  \node at (1,0) (d) {$*$};
  \draw[->] (b) -- (a) node[left,pos=.3] {$d_{1}\circ\alpha^{+}$};
  \draw[->] (b) -- (d);
  \draw[->] (c) -- (a) node[left,pos=.4] {$\widetilde{N}_{1}$};
  \draw[->] (c) -- (d);
  \draw[->] (c) -- (b) node[right,pos=.5] {$\beta$};
\end{tikzpicture}
\end{center}

Define $\widetilde{S}_{1}^{+}$ and $\widetilde{S}_{1}^{-}$ by: 
\begin{align*}
    \widetilde{S}_{1}^{+} &= \{n\in\N \mid v_{p}(n) = 0 \text{ for ramified } p,\; v_{p}(n) = 2k \text{ for inert } p\}\\
    \widetilde{S}_{1}^{-} &= \{n\in\N \mid v_{p}(n) = 0 \text{ for ramified and split } p,\; v_{p}(n) = 2k + 1 \text{ for inert } p\}. 
\end{align*}
Let $j^{+} : \widetilde{S}_{1}^{+}\hookrightarrow\widetilde{S}_{1}$ and $j^{-} : \widetilde{S}_{1}^{-}\hookrightarrow\widetilde{S}_{1}$ be the natural inclusions. Then the sets $\widetilde{B}^{+}$ and $\widetilde{B}^{-}$ appearing in the spans above are given by: 
$$
\widetilde{B}^{+} = \{\sigma\in\widetilde{S}_{2} \mid d_{0}\sigma\in\widetilde{S}_{1}^{+}\} \quad\text{and}\quad \widetilde{B}^{-} = \{\sigma\in\widetilde{S}_{2} \mid d_{0}\sigma\in\widetilde{S}_{1}^{-}\}. 
$$

We proceed to construct $\beta : \widetilde{T}_{1}\coprod\widetilde{B}^{-}\rightarrow\widetilde{B}^{+}$ and prove its required properties. First, for $\frak{a} = \prod \frak{p}^{v_{\frak{p}}(\frak{a})}$ in $\widetilde{T}_{1}$, define 
$$
\beta(\frak{a}) = \tikz[baseline=5]{
    \draw[thick] (-.7,0) -- (0,1) node[left,pos=.6] {$d_{2}\tau$};
    \draw[thick] (0,1) -- (.7,0) node[right,pos=.4] {$d_{0}\tau$};
    \draw[thick] (.7,0) -- (-.7,0) node[below,pos=.5] {$d_{1}\tau$};
    \node at (0,.4) {\large $\tau$};
}
$$
whose faces are given by: 
\begin{align*}
    d_{2}\tau &= \prod_{\frak{p}\mid p \text{ ramified}} N(\frak{p})^{v_{\frak{p}}(\frak{a})} \times \prod_{\frak{p}\mid p \text{ split}} N(\frak{p})^{v_{\frak{p}}(\frak{a})}\\
    d_{0}\tau &= \prod_{\overline{\frak{p}}\mid p \text{ split}} N(\overline{\frak{p}})^{v_{\overline{\frak{p}}}(\frak{a})} \times \prod_{\frak{p}\mid p \text{ inert}} N(\frak{p})^{v_{\frak{p}}(\frak{a})}\\
    \text{and}\quad d_{1}\tau &= N(\frak{a}) = \prod_{\frak{p}} N(\frak{p})^{v_{\frak{p}}(\frak{a})}.
\end{align*}
Meanwhile, for $\sigma\in\widetilde{B}^{-}$, define 
$$
\beta\left (\tikz[baseline=5]{
    \draw[thick] (-.7,0) -- (0,1) node[left,pos=.6] {$a$};
    \draw[thick] (0,1) -- (.7,0) node[right,pos=.4] {$b$};
    \draw[thick] (.7,0) -- (-.7,0) node[below,pos=.5] {$n$};
    \node at (0,.4) {\large $\sigma$};
}\right ) = \tikz[baseline=5]{
    \draw[thick] (-.7,0) -- (0,1) node[left,pos=.6] {$d_{2}\tau$};
    \draw[thick] (0,1) -- (.7,0) node[right,pos=.4] {$d_{0}\tau$};
    \draw[thick] (.7,0) -- (-.7,0) node[below,pos=.5] {$d_{1}\tau$};
    \node at (0,.4) {\large $\tau$};
}
$$
with faces 
\begin{align*}
    d_{2}\tau &= \prod_{p \text{ ramified}} p^{v_{p}(a)} \times \prod_{p \text{ split}} p^{v_{p}(a)} \times \prod_{p \text{ inert}} p^{v_{p}(a) + 1}\\
    d_{0}\tau &= \prod_{p \text{ inert}} p^{v_{p}(b) - 1} \quad\text{and}\quad d_{1}\tau = n = ab. 
\end{align*}
To show $\beta$ is an isomorphism, we construct its inverse $\beta^{-1} : \widetilde{B}^{+}\rightarrow\widetilde{T}_{1}\coprod\widetilde{B}^{-}$. If $\tau\in\widetilde{B}^{+}$ has $v_{p}(d_{2}\tau) = 0$ for all inert $p$, send it to $\beta^{-1}(\tau)\in\widetilde{T}_{1}$ defined by 
$$
\beta^{-1}(\tau) = \prod_{p \text{ ramified}} \frak{p}^{v_{p}(d_{1}\tau)} \times \prod_{p \text{ split}} \frak{p}^{v_{p}(d_{2}\tau)}\overline{\frak{p}}^{v_{p}(d_{0}\tau)} \times \prod_{p \text{ inert}} \frak{p}^{v_{p}(d_{1}\tau)/2}. 
$$
Otherwise, $v_{p}(d_{2}\tau) > 0$ for some inert $p$, so we send it to
$$
\beta^{-1}(\tau) = \tikz[baseline=5]{
    \draw[thick] (-.7,0) -- (0,1) node[left,pos=.6] {$d_{2}\sigma$};
    \draw[thick] (0,1) -- (.7,0) node[right,pos=.4] {$d_{0}\sigma$};
    \draw[thick] (.7,0) -- (-.7,0) node[below,pos=.5] {$d_{1}\sigma$};
    \node at (0,.4) {\large $\sigma$};
}\quad\text{in }\widetilde{B}^{-}
$$
with faces 
\begin{align*}
    d_{2}\sigma &= \prod_{p\text{ ramified}} p^{v_{p}(d_{2}\tau)} \times \prod_{p \text{ split}} p^{v_{p}(d_{2}\tau)} \times \prod_{\substack{p \text{ inert} \\ v_{p}(d_{2}\tau) > 0}} p^{v_{p}(d_{2}\tau) - 1}\\
    d_{0}\sigma &= \prod_{p\text{ ramified}} p^{v_{p}(d_{0}\tau)} \times \prod_{p \text{ split}} p^{v_{p}(d_{0}\tau)} \times \prod_{p \text{ inert}} p^{v_{p}(d_{0}\tau) + 1}\\
    \text{and}\quad d_{1}\sigma &= d_{1}\tau = \prod_{p} p^{v_{p}(d_{1}\tau)}. 
\end{align*}
Then $\beta$ is a bijection and satisfies $d_{1}\circ\beta = N\sqcup d_{1}$ by construction. This completes the proof. 
\end{proof}

\begin{rem}
\label{rem:redcard}
Taking the cardinality of the formula in Theorem~\ref{thm:mainthmred} recovers formula (\ref{eq:quadzeta}), so Theorem~\ref{thm:mainthmred} can be considered a categorification of that formula for the Dedekind zeta function of a quadratic number field. Explicitly, cardinality is a map from the abstract incidence algebra $I(\N^{\times})$ to the numerical incidence algebra $I_{\#}(\N^{\times})$: 
\begin{align*}
    |\cdot| : I(\N^{\times}) &\longrightarrow I_{\#}(\N^{\times})\\
    \varphi = \left (\roofx{S_{1}}{*}{M}{f}{}\right ) &\longmapsto \left (|\varphi| : n \mapsto |f^{-1}(n)| \right ). 
\end{align*}
Applying $|\cdot|$ to the formula $\widetilde{N}_{*}\zeta_{K} + \zeta_{\Q}*\widetilde{L}(\chi)^{-} \cong \zeta_{\Q}*\widetilde{L}(\chi)^{+}$ recovers formula (\ref{eq:quadzeta}), using the fact that $I_{\#}(\N^{\times})$ is isomorphic to the ring of Dirichlet series $DS(\Q)$ by Example~\ref{ex:divposet}. 
\end{rem}

\subsection{The Full Version}
\label{sec:mainthmfull}

As in the reduced case, we start by proving a local version of Theorem~\ref{thm:introthm} for each prime factor. For a prime $p$, let $S = (\{p^{k}\},\mid)$ be the subposet of $(\N,\mid)$ consisting of prime powers of $p$; let $T$ be the analogous subposet of $(I_{K}^{+},\mid)$ consisting of ideals above powers of $p$; let $N : T\rightarrow S$ be the restriction of the norm map to these subposets; and let $\zeta_{\Q,p}$ and $\zeta_{K,p}$ be the zeta functions in $I(S)$ and $I(T)$, respectively. 

\begin{thm}[{Local Version of Theorem~\ref{thm:introthm}}]
\label{thm:mainthmloc}
For each prime $p$, in the incidence algebra $I(\{p^{k}\},\mid)$, there is an equivalence of linear functors 
$$
N_{*}\zeta_{K,p} + \zeta_{\Q,p}*L_{p}(\chi)^{-} \cong \zeta_{\Q,p}*L_{p}(\chi)^{+}
$$
\end{thm}

\begin{proof}
We first define the terms appearing in the local formula. The first term, $N_{*}\zeta_{K,p}$, is represented by the span 
$$
N_{*}\zeta_{K,p} = \left (\tikz[baseline=10]{
  \node at (-1,0) (a) {$S_{1}$};
  \node at (1,0) (b) {$*$};
  \node at (0,1) (top) {$T_{1}$};
  \draw[->] (top) -- (a) node[above,pos=.7] {$N_{1}\;\;$};
  \draw[->] (top) -- (b);
}\right )
$$
(see Example~\ref{ex:pushfwd}). The linear functors $L_{p}(\chi)^{+}$ and $L_{p}(\chi)^{-}$ will be defined below. As with Theorem~\ref{thm:mainthmredloc}, we now divide the proof of the local formula into three cases. 

First, suppose $p$ ramifies in $\orb_{K}$. Let $S_{1}^{+}$ be the set of $1$-simplices in $S_{1}$ of the form $[p^{k},p^{k}]$ and let $j^{+} : S_{1}^{+}\rightarrow S_{1}$ be the natural inclusion. Also, let $S_{1}^{-} = \varnothing$ with the unique map $j^{-} : S_{1}^{-}\rightarrow S_{1}$. Then $\chi_{p}^{+}$ may be defined as the linear functor 
$$
L_{p}(\chi)^{+} = \left (\roofx{S_{1}}{*}{S_{1}^{+}}{j^{+}}{}\right )
$$
while $L_{p}(\chi)^{-}$ is the ``empty functor''. Then the convolutions we need are represented by spans: 
\begin{align*}
    \zeta_{\Q,p}*L_{p}(\chi)^{-} &= \left ( \tikz[xscale=3,yscale=1.7,baseline=30]{
        \node at (0,0) (a) {$S_{1}$};
        \node at (1,0) (b) {$S_{1}\times S_{1}$};
        \node at (2,0) (c) {$*$};
        \node at (.5,.8) (x) {$S_{2}$};
        \node at (1.5,.8) (y) {$S_{1}\times S_{1}^{-}$};
        \node at (1,1.6) (z) {$B^{-}$};
        \draw[->] (z) -- (x) node[left,pos=.3] {$\alpha^{-}$};
        \draw[->] (z) -- (y);
        \draw[->] (x) -- (a) node[left,pos=.3] {$d_{1}$};
        \draw[->] (x) -- (b) node[right,pos=.1] {$(d_{2},d_{0})$};
        \draw[->] (y) -- (b) node[right,pos=.8] {$id\times j^{-}$};
        \draw[->] (y) -- (c);
    }\right )\\
    \text{and}\quad \zeta_{\Q,p}*L_{p}(\chi)^{+} &= \left ( \tikz[xscale=3,yscale=1.7,baseline=30]{
        \node at (0,0) (a) {$S_{1}$};
        \node at (1,0) (b) {$S_{1}\times S_{1}$};
        \node at (2,0) (c) {$*$};
        \node at (.5,.8) (x) {$S_{2}$};
        \node at (1.5,.8) (y) {$S_{1}\times S_{1}^{+}$};
        \node at (1,1.6) (z) {$B^{+}$};
        \draw[->] (z) -- (x) node[left,pos=.3] {$\alpha^{+}$};
        \draw[->] (z) -- (y);
        \draw[->] (x) -- (a) node[left,pos=.3] {$d_{1}$};
        \draw[->] (x) -- (b) node[right,pos=.1] {$(d_{2},d_{0})$};
        \draw[->] (y) -- (b) node[right,pos=.8] {$id\times j^{+}$};
        \draw[->] (y) -- (c);
    }\right )
\end{align*}
Since $L_{p}(\chi)^{-}$ is the empty functor, the second diagram also represents the empty functor, so the local formula in the ramified case reduces to
\begin{equation}\label{eq:localram}
N_{*}\zeta_{K,p} \cong \zeta_{\Q,p}*L_{p}(\chi)^{+}.
\end{equation}
For convenience, we abbreviate the third diagram as a single span:
$$
\zeta_{\Q,p}*L_{p}(\chi)^{+} = \left (\roofx{S_{1}}{*}{B^{+}}{d_{1}\circ\alpha^{+}}{}\right ). 
$$
To prove the local formula (\ref{eq:localram}), we will construct an isomorphism $\beta : T_{1}\rightarrow B^{+}$ that makes the following diagram commute: 
\begin{center}
\begin{tikzpicture}[scale=1.3]
  \node at (-1,0) (a) {$S_{1}$};
  \node at (0,1) (b) {$B^{+}$};
  \node at (0,-1) (c) {$T_{1}$};
  \node at (1,0) (d) {$*$};
  \draw[->] (b) -- (a) node[left,pos=.3] {$d_{1}\circ\alpha^{+}$};
  \draw[->] (b) -- (d);
  \draw[->] (c) -- (a) node[left,pos=.5] {$N_{1}$};
  \draw[->] (c) -- (d);
  \draw[->] (c) -- (b) node[right,pos=.5] {$\beta$};
\end{tikzpicture}
\end{center}
Notice that $B^{+}$ is the pullback of $S_{1}\times S_{1}^{+}$ along $(d_{2},d_{0}) : S_{2}\rightarrow S_{1}\times S_{1}$, so explicitly, 
$$
B^{+} = \biggl\{\sigma\in S_{2} \;\biggr\rvert\; d_{0}\sigma = [p^{k},p^{k}] \text{ for some } k\geq 0\biggr\}. 
$$
In this case, $d_{1}\circ\alpha^{+}$ is just the restriction of $d_{1}$ to this subset $B^{+}$ of $S_{2}$. The map $\beta : T_{1}\rightarrow B^{+}$ is defined by 
$$
\beta([\frak{p}^{k},\frak{p}^{\ell}]) = \tikz[baseline=5]{
    \draw[thick] (-.7,0) -- (0,1) node[left,pos=.6] {$[p^{k},p^{\ell}]$};
    \draw[thick] (0,1) -- (.7,0) node[right,pos=.4] {$[p^{\ell},p^{\ell}]$};
    \draw[thick] (.7,0) -- (-.7,0) node[below,pos=.5] {$[p^{k},p^{\ell}]$};
    \node at (0,.4) {\large $\sigma$};
}
$$
This is clearly invertible since every $2$-simplex in $B^{+}$ is of the form shown. Moreover, it's easy to check that $\beta$ makes the diagram commute. Therefore the local formula (\ref{eq:localram}) holds in the ramified case. 

Next, if $p$ splits in $\orb_{K}$, let $S_{1}^{-} = \varnothing$ again, with $j^{-} : S_{1}^{-}\rightarrow S$ the unique map. On the other hand, let $S_{1}^{+}\subseteq S_{1}\times S_{1}$ be defined by 
$$
S_{1}^{+} = \{(x,y)\in S_{1}\times S_{1}\mid d_{1}(x,y)\not = (*,*),d_{0}x\not = *\}
$$
where $*$ denotes the initial element $p^{0}\in S_{0}$. Also, let $j^{+} : S_{1}^{+}\rightarrow S_{1}$ be the restriction of the map $j : S_{1}\times S_{1}\rightarrow S_{1},([p^{k},p^{k'}],[p^{\ell},p^{\ell'}]) \mapsto [p^{k + \ell},p^{k' + \ell'}]$ to $S_{1}^{+}$. This can be visualized as follows: 
\begin{center}
\begin{tikzpicture}
    \draw[gray] (0,0) -- (1,.5) (0,0) -- (1,-.5) -- (2,-1) -- (3,-1.5) -- (3.5,-1.75) (1,-.5) -- (2,0) (2,-1) -- (3,-.5) (3,-1.5) -- (3.5,-1.25);
    \fill (0,0) circle (.07);
    \fill (1,.5) circle (.07);
    \fill (1,-.5) circle (.07);
    \fill (2,1) circle (.07);
    \fill (2,0) circle (.07);
    \fill (2,-1) circle (.07);
    \fill (3,1.5) circle (.07);
    \fill (3,.5) circle (.07);
    \fill (3,-.5) circle (.07);
    \fill (3,-1.5) circle (.07);
    \node at (4,0) {$\cdots$};
    \draw[blue] (1,.5) -- (2,1) -- (3,1.5) (2,1) -- (3,.5) (1,.5) -- (2,0) -- (3,.5) (2,0) -- (3,-.5) (3,1.5) -- (3.5,1.75) (3,1.5) -- (3.5,1.25) (3,.5) -- (3.5,.75) (3,.5) -- (3.5,.25) (3,-.5) -- (3.5,-.25) (3,-.5) -- (3.5,-.75);
    \fill (0,-3) circle (.07);
    \fill (1,-3) circle (.07);
    \fill (2,-3) circle (.07);
    \fill (3,-3) circle (.07);
    \node at (4,-3) {$\cdots$};
    \draw (0,-3) -- (1,-3) -- (2,-3) -- (3,-3) -- (3.5,-3);
    \node[gray] at (-1.2,0) {$S_{1}\times S_{1}$};
    \node at (-1,-3) {$S_{1}$};
    \node[blue] at (1.6,1.5) {$S_{1}^{+}$};
    \draw[->] (.5,-.5) -- (.5,-2.5) node[left,pos=.5] {$j$};
    \draw[->] (1.5,-1) -- (1.5,-2.5) node[left,pos=.5] {$j$};
    \draw[->] (2.5,-1.5) -- (2.5,-2.5) node[left,pos=.5] {$j$};
\end{tikzpicture}
\end{center}
As in the ramified case, let $L_{p}(\chi)^{+}$ be the linear functor represented by the following span: 
$$
L_{p}(\chi)^{+} = \left (\roofx{S_{1}}{*}{S_{1}^{+}}{j^{+}}{}\right )
$$
The local formula we must prove is the same as formula (\ref{eq:localram}). Keeping the same notation as in the ramified case, we will construct an isomorphism $\beta : T_{1}\rightarrow B^{+}$ that makes the same diagram commute:
\begin{center}
\begin{tikzpicture}[scale=1.3]
  \node at (-1,0) (a) {$S_{1}$};
  \node at (0,1) (b) {$B^{+}$};
  \node at (0,-1) (c) {$T_{1}$};
  \node at (1,0) (d) {$*$};
  \draw[->] (b) -- (a) node[left,pos=.3] {$d_{1}\circ\alpha^{+}$};
  \draw[->] (b) -- (d);
  \draw[->] (c) -- (a) node[left,pos=.5] {$N$};
  \draw[->] (c) -- (d);
  \draw[->] (c) -- (b) node[right,pos=.5] {$\beta$};
\end{tikzpicture}
\end{center}
In this case, 
$$
B^{+} = \{(\sigma,x,(y,z))\in S_{2}\times S_{1}\times S_{1}^{+} \mid d_{2}\sigma = x,d_{0}\sigma = j^{+}(y,z)\}
$$
and $d_{1}\circ\alpha^{+}(\sigma,x,(y,z)) = d_{1}\sigma$. The map $\beta : T_{1}\rightarrow B^{+}$ is defined by 
$$
\beta\left ([\frak{p}^{k}\overline{\frak{p}}^{\ell},\frak{p}^{k'}\overline{\frak{p}}^{\ell'}]\right ) = \left (\tikz[baseline=5]{
        \draw[thick] (-.7,0) -- (0,1) node[left,pos=.6] {$[p^{k + \ell},p^{k' + \ell}]$};
        \draw[thick] (0,1) -- (.7,0) node[right,pos=.4] {$[p^{k' + \ell},p^{k' + \ell'}]$};
        \draw[thick] (.7,0) -- (-.7,0) node[below,pos=.5] {$[p^{k + \ell},p^{k' + \ell'}]$};
        \node at (0,.4) {\large $\sigma$};
    },[p^{k + \ell},p^{k' + \ell'}],\biggl ([p^{k'},p^{k'}],[p^{\ell},p^{\ell'}]\biggr )\right ). 
$$
This map is invertible, with inverse $\beta^{-1}$ given by 
$$
\beta^{-1}\left (\tikz[baseline=5]{
        \draw[thick] (-.7,0) -- (0,1) node[left,pos=.6] {$[p^{a},p^{b}]$};
        \draw[thick] (0,1) -- (.7,0) node[right,pos=.4] {$[p^{b},p^{c}]$};
        \draw[thick] (.7,0) -- (-.7,0) node[below,pos=.5] {$[p^{a},p^{c}]$};
        \node at (0,.4) {\large $\sigma$};
    },[p^{a},p^{c}],\biggl ([p^{b'},p^{c'}],[p^{b - b'},p^{c - c'}]\biggr )\right ) = [\frak{p}^{b' - b + a}\overline{\frak{p}}^{b - b'},\frak{p}^{c'}\overline{\frak{p}}^{c - c'}]. 
$$
Moreover, $\beta$ makes the required diagram commute since 
\begin{align*}
    d_{1}\circ\alpha^{+}\left (\beta\left ([\frak{p}^{k}\overline{\frak{p}}^{\ell},\frak{p}^{k'}\overline{\frak{p}}^{\ell'}]\right )\right ) &= d_{1}\left (\tikz[baseline=5]{
        \draw[thick] (-.7,0) -- (0,1) node[left,pos=.6] {$[p^{k + \ell},p^{k' + \ell}]$};
        \draw[thick] (0,1) -- (.7,0) node[right,pos=.4] {$[p^{k' + \ell},p^{k' + \ell'}]$};
        \draw[thick] (.7,0) -- (-.7,0) node[below,pos=.5] {$[p^{k + \ell},p^{k' + \ell'}]$};
        \node at (0,.4) {\large $\sigma$};
    }\right ) = [p^{k + \ell},p^{k' + \ell'}] = N\left ([\frak{p}^{k}\overline{\frak{p}}^{\ell},\frak{p}^{k'}\overline{\frak{p}}^{\ell'}]\right ). 
\end{align*}
Therefore local formula (\ref{eq:localram}) holds for $p$ split. 

Finally, if $p$ is inert in $\orb_{K}$, let $S_{1}^{+}$ be the set of ``even intervals'' in $S_{1}$: 
$$
S_{1}^{+} = \{[p^{k},p^{\ell}]\in S_{1}\mid \ell - k\in 2\N_{0}\}. 
$$
Let $j^{+} : S_{1}^{+}\rightarrow S_{1}$ be the natural inclusion and define $S_{1}^{-}$ to be the complement $S_{1}^{-} = S_{1}\smallsetminus S_{1}^{+}$, that is, 
$$
S_{1}^{-} = \{[p^{k},p^{\ell}]\in S_{1}\mid \ell - k\not\in 2\N_{0}\}. 
$$
Also let $j^{-} : S_{1}^{-}\rightarrow S_{1}$ be the natural inclusion. Each term in the local formula is represented by a span: 
$$
N_{*}\zeta_{K,p} = \left (\tikz[baseline=10]{
  \node at (-1,0) (a) {$S_{1}$};
  \node at (1,0) (b) {$*$};
  \node at (0,1) (top) {$T_{1}$};
  \draw[->] (top) -- (a) node[above,pos=.7] {$N_{1}\;\;$};
  \draw[->] (top) -- (b);
}\right ) \quad \zeta_{\Q,p}*L_{p}(\chi)^{-} = \left (\roofx{S_{1}}{*}{B^{-}}{d_{1}\circ\alpha^{-}}{}\right ) \quad \zeta_{\Q,p}*L_{p}(\chi)^{+} = \left (\roofx{S_{1}}{*}{B^{+}}{d_{1}\circ\alpha^{+}}{}\right ). 
$$
The linear functor $N_{*}\zeta_{K,p} + \zeta_{\Q,p}*L_{p}(\chi)^{-}$ is represented by the span 
$$
N_{*}\zeta_{K,p} + \zeta_{\Q,p}*L_{p}(\chi)^{-} = \left (\roofx{S_{1}}{*}{T_{1}\coprod B^{-}}{N_{1}\sqcup d_{1}\circ\alpha^{-}}{}\right ). 
$$

This time, to prove the local formula, we will construct an isomorphism $\beta : T_{1}\coprod B^{-}\rightarrow B^{+}$ that makes the following diagram commute: 
\begin{center}
\begin{tikzpicture}[scale=1.3]
  \node at (-1,0) (a) {$S_{1}$};
  \node at (0,1) (b) {$B^{+}$};
  \node at (0,-1) (c) {$T_{1}\coprod B^{-}$};
  \node at (1,0) (d) {$*$};
  \draw[->] (b) -- (a) node[left,pos=.3] {$d_{1}\circ\alpha^{+}$};
  \draw[->] (b) -- (d);
  \draw[->] (c) -- (a) node[left,pos=.5] {$N\sqcup d_{1}\circ\alpha^{-}$};
  \draw[->] (c) -- (d);
  \draw[->] (c) -- (b) node[right,pos=.5] {$\beta$};
\end{tikzpicture}
\end{center}
In this case, $B^{-}$ and $B^{+}$ have the following descriptions: 
\begin{align*}
    B^{-} &= \biggl\{\sigma\in S_{2} \;\biggr\rvert\; d_{0}\sigma = [p^{k},p^{\ell}],\ell - k\not\in 2\N_{0}\biggr\}\\
    \text{and}\quad B^{+} &= \biggl\{\sigma\in S_{2} \;\biggr\rvert\; d_{0}\sigma = [p^{k},p^{\ell}],\ell - k\in 2\N_{0}\biggr\}. 
\end{align*}
As in the ramified case, the maps $B^{+}\rightarrow S_{1}$ and $B^{-}\rightarrow S_{1}$ are both the restrictions of $d_{1}$. We construct $\beta : T_{1}\coprod B^{-}\rightarrow B^{+}$ as follows. For $[p^{k},p^{\ell}]\in T_{1}$, let 
$$
\beta\biggl ([p^{k},p^{\ell}]\biggr ) = \tikz[baseline=5]{
        \draw[thick] (-.7,0) -- (0,1) node[left,pos=.6] {$[p^{2k},p^{2k}]$};
        \draw[thick] (0,1) -- (.7,0) node[right,pos=.4] {$[p^{2k},p^{2\ell}]$};
        \draw[thick] (.7,0) -- (-.7,0) node[below,pos=.5] {$[p^{2k},p^{2\ell}]$};
        \node at (0,.4) {\large $\tau$};
    }
$$
On the other hand, for $\sigma\in B^{-}$, let 
$$
\beta\left (\tikz[baseline=5]{
        \draw[thick] (-.7,0) -- (0,1) node[left,pos=.6] {$[p^{k},p^{\ell}]$};
        \draw[thick] (0,1) -- (.7,0) node[right,pos=.4] {$[p^{\ell},p^{m}]$};
        \draw[thick] (.7,0) -- (-.7,0) node[below,pos=.5] {$[p^{k},p^{m}]$};
        \node at (0,.4) {\large $\sigma$};
    }\right ) = \tikz[baseline=5]{
        \draw[thick] (-.7,0) -- (0,1) node[left,pos=.6] {$[p^{k},p^{\ell + 1}]$};
        \draw[thick] (0,1) -- (.7,0) node[right,pos=.4] {$[p^{\ell + 1},p^{m}]$};
        \draw[thick] (.7,0) -- (-.7,0) node[below,pos=.5] {$[p^{k},p^{m}]$};
        \node at (0,.4) {$\beta(\sigma)$};
    }
$$
The diagram commutes since for $[p^{k},p^{\ell}]\in T_{1}$ and $\sigma\in B^{-}$, 
\begin{align*}
  d_{1}(\beta([p^{k},p^{\ell}])) &= d_{1}(\tau) = [p^{2k},p^{2\ell}] = N([p^{k},p^{\ell}])\\
  \text{and}\quad d_{1}(\beta(\sigma)) &= [p^{k},p^{m}] = d_{1}(\sigma). 
\end{align*}
Finally, we construct $\beta^{-1} : B^{+}\rightarrow T_{1}\coprod B^{-}$ as follows. If $\tau\in B^{+}$ has degenerate $2$-face $d_{2}\tau$, we send $\tau$ to $d_{1}\tau\in T_{1}$: 
$$
\beta^{-1}\left (\tikz[baseline=5]{
        \draw[thick] (-.7,0) -- (0,1) node[left,pos=.6] {$[p^{k},p^{k}]$};
        \draw[thick] (0,1) -- (.7,0) node[right,pos=.4] {$[p^{k},p^{\ell}]$};
        \draw[thick] (.7,0) -- (-.7,0) node[below,pos=.5] {$[p^{k},p^{\ell}]$};
        \node at (0,.4) {\large $\tau$};
    }\right ) = [p^{k},p^{\ell}]. 
$$
On the other hand, every other $\tau\in B^{+}$ has $d_{2}\tau = [p^{k},p^{\ell}]$ for some even $\ell - k > 0$. We send these $\tau$ to the following $2$-simplex in $B^{-}$: 
$$
\beta^{-1}\left (\tikz[baseline=5]{
        \draw[thick] (-.7,0) -- (0,1) node[left,pos=.6] {$[p^{k},p^{\ell}]$};
        \draw[thick] (0,1) -- (.7,0) node[right,pos=.4] {$[p^{\ell},p^{m}]$};
        \draw[thick] (.7,0) -- (-.7,0) node[below,pos=.5] {$[p^{k},p^{m}]$};
        \node at (0,.4) {\large $\tau$};
    }\right ) = \tikz[baseline=5]{
        \draw[thick] (-.7,0) -- (0,1) node[left,pos=.6] {$[p^{k},p^{\ell - 1}]$};
        \draw[thick] (0,1) -- (.7,0) node[right,pos=.4] {$[p^{\ell - 1},p^{m}]$};
        \draw[thick] (.7,0) -- (-.7,0) node[below,pos=.5] {$[p^{k},p^{m}]$};
        \node at (0,.4) {\large $\tau$};
    }
$$
It's clear that $\beta$ and $\beta^{-1}$ are inverses. This completes the proof of the local formula in every case. 
\end{proof}

\begin{rem}
As in Remark~\ref{rem:splitnotcanonical}, in the split case, the isomorphism $N_{*}\zeta_{K,p}\cong \zeta_{\Q,p}*L_{p}(\chi)^{+}$ is \emph{not canonical} since it depends on the choice of $\frak{p}$ and $\overline{\frak{p}}$ lying over $p$.
\end{rem}

We now prove Theorem~\ref{thm:introthm}, which is both an analogue of Theorem~\ref{thm:introred} for the full incidence algebra of $S = (\N,\mid)$ and a global version of Theorem~\ref{thm:mainthmloc}. 

\begin{thm}[{Theorem~\ref{thm:introthm}}]
\label{thm:mainthm}
In the incidence algebra $I(S) = I(\N,\mid)$, there is an equivalence of linear functors 
$$
N_{*}\zeta_{K} + \zeta_{\Q}*L(\chi)^{-} \cong \zeta_{\Q}*L(\chi)^{+}. 
$$
\end{thm}

\begin{proof}
As before, the terms in the global formula will be represented by spans 
\begin{align*}
    N_{*}\zeta_{K} &= \left (\tikz[baseline=10]{
      \node at (-1,0) (a) {$S_{1}$};
      \node at (1,0) (b) {$*$};
      \node at (0,1) (top) {$T_{1}$};
      \draw[->] (top) -- (a) node[above,pos=.7] {$N_{1}\;\;$};
      \draw[->] (top) -- (b);
    }\right )\\
    \zeta_{\Q}*L(\chi)^{-} &= \left ( \tikz[xscale=3,yscale=1.7,baseline=30]{
        \node at (0,0) (a) {$S_{1}$};
        \node at (1,0) (b) {$S_{1}\times S_{1}$};
        \node at (2,0) (c) {$*$};
        \node at (.5,.8) (x) {$S_{2}$};
        \node at (1.5,.8) (y) {$S_{1}\times S_{1}^{-}$};
        \node at (1,1.6) (z) {$B^{-}$};
        \draw[->] (z) -- (x) node[left,pos=.3] {$\alpha^{-}$};
        \draw[->] (z) -- (y);
        \draw[->] (x) -- (a) node[left,pos=.3] {$d_{1}$};
        \draw[->] (x) -- (b) node[right,pos=.1] {$(d_{2},d_{0})$};
        \draw[->] (y) -- (b) node[right,pos=.8] {$id\times j^{-}$};
        \draw[->] (y) -- (c);
    }\right )\\
    \text{and}\quad \zeta_{\Q}*L(\chi)^{+} &= \left ( \tikz[xscale=3,yscale=1.7,baseline=30]{
        \node at (0,0) (a) {$S_{1}$};
        \node at (1,0) (b) {$S_{1}\times S_{1}$};
        \node at (2,0) (c) {$*$};
        \node at (.5,.8) (x) {$S_{2}$};
        \node at (1.5,.8) (y) {$S_{1}\times S_{1}^{+}$};
        \node at (1,1.6) (z) {$B^{+}$};
        \draw[->] (z) -- (x) node[left,pos=.3] {$\alpha^{+}$};
        \draw[->] (z) -- (y);
        \draw[->] (x) -- (a) node[left,pos=.3] {$d_{1}$};
        \draw[->] (x) -- (b) node[right,pos=.1] {$(d_{2},d_{0})$};
        \draw[->] (y) -- (b) node[right,pos=.8] {$id\times j^{+}$};
        \draw[->] (y) -- (c);
    }\right )
\end{align*}
We need to define $j^{+} : S_{1}^{+}\rightarrow S_{1}$ and $j^{-} : S_{1}^{-}\rightarrow S_{1}$ and construct an isomorphism $\beta : T_{1}\coprod B^{-}\rightarrow B^{+}$ making the following diagram commute:
\begin{center}
\begin{tikzpicture}[scale=1.3]
  \node at (-1,0) (a) {$S_{1}$};
  \node at (0,1) (b) {$B^{+}$};
  \node at (0,-1) (c) {$T_{1}\coprod B^{-}$};
  \node at (1,0) (d) {$*$};
  \draw[->] (b) -- (a) node[left,pos=.3] {$d_{1}\circ\alpha^{+}$};
  \draw[->] (b) -- (d);
  \draw[->] (c) -- (a) node[left,pos=.4] {$N_{1}\sqcup d_{1}\circ\alpha^{-}$};
  \draw[->] (c) -- (d);
  \draw[->] (c) -- (b) node[right,pos=.5] {$\beta$};
\end{tikzpicture}
\end{center}

Define $S_{1}^{+}$ and $S_{1}^{-}$ by: 
\begin{align*}
    S_{1}^{+} &= \left\{([a,b],[a',b'])\in S_{1}\times S_{1} \left\lvert \begin{matrix}
        & v_{p}(a) = v_{p}(b),v_{p}(a') = v_{p}(b') = 0 \text{ for ramified } p\\
        & v_{p}(a)v_{p}(a')\not = 0,v_{p}(b)\not = 0 \text{ for split } p \phantom{aaaaaaaaaa}\\
        & v_{p}(b) - v_{p}(a) = 2k,v_{p}(a') = v_{p}(b') = 0 \text{ for inert } p
    \end{matrix}
    \right.\right\}\\
    S_{1}^{-} &= \{[a,b]\in S_{1} \mid v_{p}(a) = v_{p}(b) \text{ for ramified and split } p,\; v_{p}(b) - v_{p}(a) = 2k + 1 \text{ for inert } p\}. 
\end{align*}
Let $j^{-} : S_{1}^{-}\hookrightarrow S_{1}$ be the natural inclusion and let $j^{+} : S_{1}^{+}\rightarrow S_{1}$ be the restriction of the map $S_{1}\times S_{1}\rightarrow S_{1},([a,b],[a',b'])\mapsto [aa',bb']$. Then the sets $B^{+}$ and $B^{-}$ are given by: 
$$
B^{+} = \{(\sigma,x,(y,z))\in S_{2}\times S_{1}\times S_{1}^{+} \mid d_{2}\sigma = x,d_{0}\sigma = j^{+}(y,z)\} \quad\text{and}\quad B^{-} = \{\sigma\in S_{2} \mid d_{0}\sigma\in S_{1}^{-}\}. 
$$

Next, we construct the isomorphism $\beta : T_{1}\coprod B^{-}\rightarrow B^{+}$. For $[\frak{a},\frak{c}]\in T_{1}$, set 
$$
\beta\left ([\frak{a},\frak{c}]\right ) = \left (\tikz[baseline=5]{
        \draw[thick] (-.7,0) -- (0,1) node[left,pos=.6] {$[a,b]$};
        \draw[thick] (0,1) -- (.7,0) node[right,pos=.4] {$[b,c]$};
        \draw[thick] (.7,0) -- (-.7,0) node[below,pos=.5] {$[a,c]$};
        \node at (0,.4) {\large $\tau$};
    },[a,b],\biggl ([b',c'],[b'',c'']\biggr )\right )
$$
where  
\begin{align*}
    a &= \prod_{p\text{ ramified}} p^{v_{\frak{p}}(\frak{a})} \times \prod_{p\text{ split}} p^{v_{\frak{p}}(\frak{a}) + v_{\overline{\frak{p}}}(\frak{a})} \times \prod_{p\text{ inert}} p^{2v_{\frak{p}}(\frak{a})} = N(\frak{a})\\
    b &= \prod_{p\text{ ramified}} p^{v_{\frak{p}}(\frak{c})} \times \prod_{p\text{ split}} p^{v_{\frak{p}}(\frak{c}) + v_{\overline{\frak{p}}}(\frak{a})} \times \prod_{p\text{ inert}} p^{2v_{\frak{p}}(\frak{a})}\\
    c &= \prod_{p\text{ ramified}} p^{v_{\frak{p}}(\frak{c})} \times \prod_{p\text{ split}} p^{v_{\frak{p}}(\frak{c}) + v_{\overline{\frak{p}}}(\frak{c})} \times \prod_{p\text{ inert}} p^{2v_{\frak{p}}(\frak{c})} = N(\frak{c})\\
    b' &= \prod_{p\text{ ramified}} p^{v_{\frak{p}}(\frak{c})} \times \prod_{p\text{ split}} p^{v_{\frak{p}}(\frak{c})} \times \prod_{p\text{ inert}} p^{2v_{\frak{p}}(\frak{a})}\\
    c' &= \prod_{p\text{ ramified}} p^{v_{\frak{p}}(\frak{c})} \times \prod_{p\text{ split}} p^{v_{\frak{p}}(\frak{c})} \times \prod_{p\text{ inert}} p^{2v_{\frak{p}}(\frak{c})}\\
    b'' &= \prod_{p\text{ split}} p^{v_{\overline{\frak{p}}}(\frak{a})} \quad\text{and}\quad c'' = \prod_{p\text{ split}} p^{v_{\overline{\frak{p}}}(\frak{c})}. 
\end{align*}
On the other hand, for $\sigma\in B^{-}$, set 
$$
\beta\left (\tikz[baseline=5]{
        \draw[thick] (-.7,0) -- (0,1) node[left,pos=.6] {$[d,e]$};
        \draw[thick] (0,1) -- (.7,0) node[right,pos=.4] {$[e,f]$};
        \draw[thick] (.7,0) -- (-.7,0) node[below,pos=.5] {$[d,f]$};
        \node at (0,.4) {\large $\sigma$};
    }\right ) = \left (\tikz[baseline=5]{
        \draw[thick] (-.7,0) -- (0,1) node[left,pos=.6] {$[d,e']$};
        \draw[thick] (0,1) -- (.7,0) node[right,pos=.4] {$[e',f]$};
        \draw[thick] (.7,0) -- (-.7,0) node[below,pos=.5] {$[d,f]$};
        \node at (0,.4) {\large $\tau$};
    },[d,e'],([e',f],[1,1])\right )
$$
where 
$$
e' = \prod_{p\text{ ramified}} p^{v_{p}(e)} \times \prod_{p\text{ split}} p^{v_{p}(e)} \times \prod_{p\text{ inert}} p^{v_{p}(e) + 1}. 
$$
This defines $\beta$ on all of $T_{1}\coprod B^{-}$ and it is clear that $d_{1}\circ\alpha^{+}\circ\beta = N_{1}\sqcup d_{1}\circ\alpha^{-}$. An inverse can be constructed as follows. For
$$
\left (\tikz[baseline=5]{
        \draw[thick] (-.7,0) -- (0,1) node[left,pos=.6] {$[a,b]$};
        \draw[thick] (0,1) -- (.7,0) node[right,pos=.4] {$[b,c]$};
        \draw[thick] (.7,0) -- (-.7,0) node[below,pos=.5] {$[a,c]$};
        \node at (0,.4) {\large $\tau$};
    },[a,b],([b',c'],[b'',c''])\right )\in B^{+}
$$
with $v_{p}(a) = v_{p}(b)$ for all inert $p$, set 
$$
\beta^{-1}\left (\tikz[baseline=5]{
        \draw[thick] (-.7,0) -- (0,1) node[left,pos=.6] {$[a,b]$};
        \draw[thick] (0,1) -- (.7,0) node[right,pos=.4] {$[b,c]$};
        \draw[thick] (.7,0) -- (-.7,0) node[below,pos=.5] {$[a,c]$};
        \node at (0,.4) {\large $\tau$};
    },[a,b],([b',c'],[b'',c''])\right ) = [\frak{a},\frak{c}] \in T_{1}
$$
where 
\begin{align*}
    \frak{a} &= \prod_{p\mid a \text{ ramified}} \frak{p}^{v_{p}(a)} \times \prod_{p\mid a \text{ split}} \frak{p}^{v_{p}(b') - v_{p}(b) + v_{p}(a)}\overline{\frak{p}}^{v_{p}(b) - v_{p}(b')} \times \prod_{p\mid a \text{ inert}} \frak{p}^{v_{p}(a)}\\
    \text{and}\quad \frak{c} &= \prod_{p\mid c \text{ ramified}} \frak{p}^{v_{p}(c)} \times \prod_{p\mid c \text{ split}} \frak{p}^{v_{p}(c')}\overline{\frak{p}}^{v_{p}(c) - v_{p}(c')} \times \prod_{p\mid c \text{ inert}} \frak{p}^{v_{p}(c)}. 
\end{align*}
Otherwise, $v_{p}(b) > v_{p}(a)$ for some inert prime $p$ and we put 
$$
\beta^{-1}\left (\tikz[baseline=5]{
        \draw[thick] (-.7,0) -- (0,1) node[left,pos=.6] {$[a,b]$};
        \draw[thick] (0,1) -- (.7,0) node[right,pos=.4] {$[b,c]$};
        \draw[thick] (.7,0) -- (-.7,0) node[below,pos=.5] {$[a,c]$};
        \node at (0,.4) {\large $\tau$};
    },[a,b],([b',c'],[b'',c''])\right ) = \tikz[baseline=5]{
        \draw[thick] (-.7,0) -- (0,1) node[left,pos=.6] {$[a,d]$};
        \draw[thick] (0,1) -- (.7,0) node[right,pos=.4] {$[d,c]$};
        \draw[thick] (.7,0) -- (-.7,0) node[below,pos=.5] {$[a,c]$};
        \node at (0,.4) {\large $\sigma$};
    } \in B^{-}
$$
where 
$$
d = \prod_{p\mid b \text{ ramified}} p^{v_{p}(b)} \times \prod_{p\mid b \text{ split}} p^{v_{p}(b)} \times \prod_{p\mid b \text{ inert}} p^{v_{p}(b) - 1}. 
$$
This finishes the proof. 
\end{proof}

\begin{rem}
\label{rem:fullcard}
While Remark~\ref{rem:redcard} showed that taking the cardinality of Theorem~\ref{thm:mainthmred} recovers the familiar formula (\ref{eq:quadzeta}) for the Dedekind zeta function, there is no previously known numerical version of Theorem~\ref{thm:mainthm} as far as we could find. Nevertheless, applying the cardinality map $|\cdot| : I(\N,\mid)\rightarrow I_{\#}(\N,\mid)$ produces the following numerical formula:
$$
N_{*}\zeta_{K} = \zeta_{\Q}*(L(\chi)^{+} - L(\chi)^{-}). 
$$
More specifically, for any interval $[a,b]\subset (\N,\mid)$, 
$$
\biggl\lvert\biggl\{[\frak{a},\frak{b}]\subset (I_{K}^{+},\mid) : N(\frak{a}) = a,N(\frak{b}) = b\biggr\}\biggr\rvert = \sum_{a\mid d\mid b} \chi([d,b])
$$
where $\chi\in I_{\#}(\N,\mid)$ is the function 
$$
\chi([m,n]) = |(j^{+})^{-1}([m,n])| - |(j^{-})^{-1}([m,n])|. 
$$
\end{rem}

\subsection{Relation Between the Theorems}
\label{sec:reduction}

In this section, we use the notion of decalage (Section~\ref{sec:decompsets}) to pass from $I(\N,\mid)$ to $\widetilde{I}(\N,\mid) \cong I(\N^{\times})$ and compare Theorems~\ref{thm:mainthmred} and~\ref{thm:mainthm}. Recall that the identification $\widetilde{I}(\N,\mid) \cong I(\N^{\times})$ is achieved through an algebra homomorphism $I(\N^{\times})\hookrightarrow I(\N,\mid)$ which is in turn induced by a map of decomposition sets $(\N,\mid)\rightarrow\N^{\times}$ (see Example~\ref{ex:divposetdec} or \cite[Sec.~1.5.3]{gkt5}). It is also possible, by a similar technique, to pass from the full incidence algebra $I(\N,\mid)$ to the reduced subalgebra $\widetilde{I}(\N,\mid)$.

\begin{lem}
On the level of abstract incidence algebras, there is linear functor $I(\N,\mid)\rightarrow I(\N^{\times})$ which is the identity on $\widetilde{I}(\N,\mid) \cong I(\N^{\times})$. 
\end{lem}

\begin{proof}
Following the notation in the proofs of Theorems~\ref{thm:mainthmred} and~\ref{thm:mainthm}, set $S = (\N,\mid)$ and $\widetilde{S} = \N^{\times}$. We start by constructing a linear functor in the other direction: $\varphi : C(\widetilde{S})\rightarrow C(S)$. Consider the span 
$$
\roofx{\widetilde{S}_{1}}{S_{1}}{\widetilde{S}_{1}}{\operatorname{id}}{\gamma}
$$
where $\gamma : n\mapsto [1,n]$. This induces the linear functor 
\begin{align*}
    \varphi : C(\widetilde{S}) &\longrightarrow C(S)\\
      (X\rightarrow \widetilde{S}_{1}) &\longmapsto (X\rightarrow \widetilde{S}_{1}\xrightarrow{\gamma} S_{1}). 
\end{align*}
Dualizing produces the desired linear functor $\varphi^{*} : I(S)\rightarrow I(\widetilde{S})$. Moreover, it's easy to check that the composition $I(\widetilde{S})\hookrightarrow I(S)\xrightarrow{\varphi^{*}} I(\widetilde{S})$ is the identity (see below). 
\end{proof}

Explicitly, for $f\in I(S)$ represented by a span 
$$
f = \left (\roof{S}{*}{F}\right )
$$
$\varphi^{*}f$ is the functor represented by 
$$
\varphi^{*}f = \left ( \tikz[xscale=3,yscale=1.7,baseline=30]{
        \node at (0,0) (a) {$\widetilde{S}_{1}$};
        \node at (1,0) (b) {$S_{1}$};
        \node at (2,0) (c) {$*$};
        \node at (.5,.8) (x) {$\widetilde{S}_{1}$};
        \node at (1.5,.8) (y) {$F$};
        \node at (1,1.6) (z) {$C$};
        \draw[->] (z) -- (x);
        \draw[->] (z) -- (y);
        \draw[->] (x) -- (a) node[left,pos=.3] {$\operatorname{id}$};
        \draw[->] (x) -- (b) node[right,pos=.3] {$\gamma$};
        \draw[->] (y) -- (b);
        \draw[->] (y) -- (c);
    }\right ). 
$$
At the level of numerical incidence algebras, this is the linear map taking $f\in I_{\#}(\N,\mid)$ to $\varphi^{*}f(n) = f([1,n])$. One can think of this operation as ``only remember the values of $f$ on the initial intervals $[1,n]$''. 

\begin{thm}
\label{thm:reduction}
The formula in Theorem~\ref{thm:mainthmred} is the image in $I(\N^{\times})$ of the formula in Theorem~\ref{thm:mainthm} under the linear functor $\varphi^{*} : I(\N,\mid)\rightarrow I(\N^{\times})$. 
\end{thm}

\begin{proof}
First, $\varphi^{*}$ is linear so the sum on the left side of the formula is preserved. We need only check that $\varphi^{*}$ takes each term in the formula in $I(\N,\mid)$ to its counterpart in $I(\N^{\times})$. For $N_{*}\zeta_{K}$, we have 
$$
\varphi^{*}N_{*}\zeta_{K} = \left ( \tikz[xscale=3,yscale=1.7,baseline=30]{
        \node at (0,0) (a) {$\widetilde{S}_{1}$};
        \node at (1,0) (b) {$S_{1}$};
        \node at (2,0) (c) {$*$};
        \node at (.5,.8) (x) {$\widetilde{S}_{1}$};
        \node at (1.5,.8) (y) {$S_{1}$};
        \node at (1,1.6) (z) {$C$};
        \draw[->] (z) -- (x);
        \draw[->] (z) -- (y);
        \draw[->] (x) -- (a) node[left,pos=.3] {$\operatorname{id}$};
        \draw[->] (x) -- (b) node[right,pos=.3] {$\gamma$};
        \draw[->] (y) -- (b) node[left,pos=.3] {$\operatorname{id}$};
        \draw[->] (y) -- (c);
    }\right ) = \left (\roofx{\widetilde{S}_{1}}{*}{\widetilde{S}_{1}}{\operatorname{id}}{}\right ) = \widetilde{N}_{*}\zeta_{K}. 
$$
For $\zeta_{\Q}*L(\chi)^{-}$, we have
$$
\varphi^{*}(\zeta_{\Q}*L(\chi)^{-}) = \left ( \tikz[xscale=3,yscale=1.7,baseline=30]{
        \node at (0,0) (a) {$\widetilde{S}_{1}$};
        \node at (1,0) (b) {$S_{1}$};
        \node at (2,0) (c) {$*$};
        \node at (.5,.8) (x) {$\widetilde{S}_{1}$};
        \node at (1.5,.8) (y) {$B^{-}$};
        \node at (1,1.6) (z) {$C^{-}$};
        \draw[->] (z) -- (x);
        \draw[->] (z) -- (y);
        \draw[->] (x) -- (a) node[left,pos=.3] {$\operatorname{id}$};
        \draw[->] (x) -- (b) node[right,pos=.3] {$\gamma$};
        \draw[->] (y) -- (b) node[left,pos=.3] {$d_{1}$};
        \draw[->] (y) -- (c);
    }\right )
$$
so to verify that $\varphi^{*}(\zeta_{\Q}*L(\chi)^{-}) \cong \zeta_{\Q}*\widetilde{L}(\chi)^{-}\in I(\widetilde{S})$, it's enough to check that $C^{-}$ is isomorphic to $\widetilde{B}^{-}$ from the proof of Theorem~\ref{thm:mainthmred}. The maps 
\begin{align*}
    C^{-} &\longrightarrow \widetilde{S}_{2}, \quad \tikz[baseline=5]{
        \draw[thick] (-.7,0) -- (0,1) node[left,pos=.6] {$[1,b]$};
        \draw[thick] (0,1) -- (.7,0) node[right,pos=.4] {$[b,c]$};
        \draw[thick] (.7,0) -- (-.7,0) node[below,pos=.5] {$[1,c]$};
        \node at (0,.4) {\large $\sigma$};
    } \longmapsto \tikz[baseline=5]{
        \draw[thick] (-.7,0) -- (0,1) node[left,pos=.6] {$b$};
        \draw[thick] (0,1) -- (.7,0) node[right,pos=.4] {$\frac{c}{b}$};
        \draw[thick] (.7,0) -- (-.7,0) node[below,pos=.5] {$c$};
        \node at (0,.4) {\large $\tau$};
    }\\
    \text{and}\quad C^{-} &\longrightarrow \widetilde{S}_{1}\times\widetilde{S}_{1}^{-}, \quad \tikz[baseline=5]{
        \draw[thick] (-.7,0) -- (0,1) node[left,pos=.6] {$[1,b]$};
        \draw[thick] (0,1) -- (.7,0) node[right,pos=.4] {$[b,c]$};
        \draw[thick] (.7,0) -- (-.7,0) node[below,pos=.5] {$[1,c]$};
        \node at (0,.4) {\large $\sigma$};
    } \longmapsto \left (b,\frac{c}{b}\right )
\end{align*}
induce a map $C^{-}\rightarrow\widetilde{B}^{-}$. Conversely, 
\begin{align*}
    \widetilde{B}^{-} &\longrightarrow \widetilde{S}_{1}, \quad \tikz[baseline=5]{
        \draw[thick] (-.7,0) -- (0,1) node[left,pos=.6] {$a$};
        \draw[thick] (0,1) -- (.7,0) node[right,pos=.4] {$b$};
        \draw[thick] (.7,0) -- (-.7,0) node[below,pos=.5] {$c$};
        \node at (0,.4) {\large $\sigma$};
    } \longmapsto c\\
    \text{and}\quad \widetilde{B}^{-} &\longrightarrow B^{-}, \quad \tikz[baseline=5]{
        \draw[thick] (-.7,0) -- (0,1) node[left,pos=.6] {$a$};
        \draw[thick] (0,1) -- (.7,0) node[right,pos=.4] {$b$};
        \draw[thick] (.7,0) -- (-.7,0) node[below,pos=.5] {$c$};
        \node at (0,.4) {\large $\sigma$};
    } \longmapsto \tikz[baseline=5]{
        \draw[thick] (-.7,0) -- (0,1) node[left,pos=.6] {$[1,a]$};
        \draw[thick] (0,1) -- (.7,0) node[right,pos=.4] {$[a,ab]$};
        \draw[thick] (.7,0) -- (-.7,0) node[below,pos=.5] {$[1,c]$};
        \node at (0,.4) {\large $\tau$};
    }
\end{align*}
induce its inverse $\widetilde{B}^{-}\rightarrow C^{-}$. This shows $\varphi^{*}(\zeta_{\Q}*L(\chi)^{-}) \cong \zeta_{\Q}*\widetilde{L}(\chi)^{-}$ as required. A similar proof shows $\varphi^{*}(\zeta_{\Q}*L(\chi)^{+}) \cong \zeta_{\Q}*\widetilde{L}(\chi)^{+}$. 
\end{proof}

\begin{rem}
Although $\varphi^{*}$ preserves the products $\zeta_{\Q}*L(\chi)^{-}$ and $\zeta_{\Q}*L(\chi)^{+}$, it \emph{does not} preserve products in general. However, for $f\in I(S)$ and $g\in\widetilde{I}(S)$, $\varphi^{*}(f*g) \cong \varphi^{*}f*\varphi^{*}g$. 
\end{rem}

\subsection{General Quadratic Extension}
\label{sec:genquad}

The proofs of Theorems~\ref{thm:mainthmred} and~\ref{thm:mainthm} generalize easily to arbitrary quadratic extensions of number fields, although interpreting these formulas in terms of Dirichlet series requires an extra step. 

Let $F/\Q$ be any number field and let $K/F$ be a quadratic extension. By class field theory, there is a quadratic Hecke character $\chi : I_{F}\rightarrow\C^{\times}$ attached to this extension. One can check that it satisfies 
$$
\chi(p) = \begin{cases}
    0, &\text{if $p$ ramifies in } K\\
    1, &\text{if $p$ splits in } K\\
    -1, &\text{if $p$ is inert in } K
\end{cases}
$$
for prime ideals $p$ in $\orb_{F}$. Set $\widetilde{S} = I_{F}^{\times},\widetilde{T} = I_{K}^{\times},S = (I_{F}^{+},\mid)$ and $T = (I_{K}^{+},\mid)$. Also let $\widetilde{N} : \widetilde{T}\rightarrow\widetilde{S}$ and $N : T\rightarrow S$ be the simplicial maps induced by the ideal norm $N_{K/F}$ and let $\widetilde{N}_{*}$ and $N_{*}$ the pushforwards on the respective incidence algebras. 

\begin{thm}
\label{thm:mainthmarb}
Let $K/F$ be a quadratic extension of number fields. 
\begin{enumerate}[(1)]
\item In the reduced incidence algebra $\widetilde{I}(S) \cong I(\widetilde{S})$, there is an equivalence of linear functors 
$$
\widetilde{N}_{*}\zeta_{K} + \zeta_{F}*\widetilde{L}(\chi)^{-} \cong \zeta_{F}*\widetilde{L}(\chi)^{+}. 
$$

\item In the full incidence algebra $I(S)$, there is an equivalence of linear functors 
$$
N_{*}\zeta_{K} + \zeta_{F}*L(\chi)^{-} \cong \zeta_{F}*L(\chi)^{+}. 
$$
\end{enumerate}
\end{thm}

\begin{proof}
Identical to the proofs of Theorems~\ref{thm:mainthmred} and~\ref{thm:mainthm}. 
\end{proof}

\begin{cor}
In the numerical incidence algebra $I_{\#}(I_{F}^{\times})$, there is a formula 
$$
\widetilde{N}_{*}\zeta_{K} \cong \zeta_{F}*\widetilde{L}(\chi)
$$
where $\widetilde{L}(\chi) = \widetilde{L}(\chi)^{+} - \widetilde{L}(\chi)^{-}$. Furthermore, applying the pushforward along $N_{K/\Q}$ yields the following formula of Dirichlet series: 
$$
\zeta_{K}(s) = \zeta_{F}(s)L(\chi,s)
$$
where $L(\chi,s)$ is the $L$-function of the Hecke character $\chi$. 
\end{cor}

\begin{proof}
The first statement follows by applying cardinality $|\cdot| : I(I_{F}^{\times})\rightarrow I_{\#}(I_{F}^{\times})$ to the formula in Theorem~\ref{thm:mainthmarb}(1). The second statement can be obtained by pushing forward along $N_{K/\Q}$ and identifying $L(\chi,s)$ with the pushforward of $L(\chi)$. 
\end{proof}

\begin{rem}
There is a similar result in the full numerical incidence algebra $I_{\#}(I_{F}^{+},\mid)$, though it doesn't have an interpretation in terms of Dirichlet series (see Remark~\ref{rem:fullcard}). 
\end{rem}

%% file: relzeta.tex
In this section, we introduce a notion of relative zeta function in an abstract incidence algebra. This term is intended to generalize the following situation. For a quadratic extension $K/\Q$, equation (\ref{eq:quadzeta}) can be rewritten as 
$$
L(\chi,s) = \frac{\zeta_{K}(s)}{\zeta_{\Q}(s)}. 
$$
We should think of the $L$-function $L(\chi,s)$ as a ``zeta function relative to the extension $K/\Q$''. This notion is formalized in the next definition. 

\begin{defn}
Let $S$ be a decomposition set. A {\bf relative zeta function} for $S$ is an element $f\in I(S)$ such that $\zeta_{S}*f \cong \varphi_{*}\zeta_{T}$ for some decomposition set $T$ and some simplicial map $\varphi : T\rightarrow S$. 
\end{defn}

Replacing $I(S)$ with the numerical incidence algebra $I_{\#}(S)$, we can also talk about numerical relative zeta functions. In general, a relative zeta function ``cuts out a simplicial map'' $T\rightarrow S$, much as $L(\chi,s)$ above cuts out the quadratic extension $K/\Q$. 

\begin{ex}
The unit $\delta\in I(S)$ is a relative zeta function, since $\zeta_{S}*\delta = \zeta_{S} = id_{*}\zeta_{S}$. That is, $\delta$ cuts out the identity map $S\rightarrow S$. 
\end{ex}

\begin{ex}
\label{ex:relMob}
Suppose $S_{0} = \{x\}$. In the numerical incidence algebra $I_{\#}(S)$, the M\"{o}bius function $\mu_{S}$ is the inverse of $\zeta_{S}$, so $\mu_{S}$ is a relative zeta function: $\zeta_{S}*\mu_{S} = \delta = \iota_{*}\zeta_{\bullet}$ where $\bullet$ is the trivial simplicial space (a unique simplex in every dimension) and $\iota : \bullet\hookrightarrow S$ is the simplicial map taking $*\in\bullet_{0}$ to the unique $x\in S_{0}$ and $*\in\bullet_{k}$ to $s_{0}^{k}(x)\in S_{k}$ for each $k\geq 1$. In other words, $\mu_{S}$ cuts out the degenerate simplices $\{s_{0}^{k}(x)\}$ as a simplicial subset of $S$. In the abstract incidence algebra $I(S)$, $\zeta_{S}$ need not be invertible. Instead, the authors in \cite{gkt2} prove a sign-free version of M\"{o}bius inversion \cite[Thm.~3.8]{gkt2}: 
$$
\delta + \zeta_{S}*\mu_{1} \cong \zeta_{S}*\mu_{0}
$$
which becomes 
$$
\delta = \zeta_{S}*(\mu_{0} - \mu_{1}) = \zeta_{S}*\mu_{S}
$$
in the numerical incidence algebra. Therefore we can interpret $\mu_{0}$ and $\mu_{1}$ (denoted $\Phi_{\operatorname{even}}$ and $\Phi_{\operatorname{odd}}$ in {\it loc.~cit.}) as an additive decomposition of a numerical relative zeta function. We don't have a good name for this phenomenon, but the main results in Section~\ref{sec:mainthm} show that it is more prevalent than just the M\"{o}bius inversion principle. 
\end{ex}

\begin{ex}
Let $\widetilde{S} = \N^{\times}$ and fix a prime integer $p$. Write $\widetilde{S} = \widetilde{S}_{p}\times\prod_{\ell\not = p} \widetilde{S}_{\ell}$ where $\widetilde{S}_{\ell} = \{\ell^{k}\}^{\times}$ as a submonoid of $\widetilde{S}$. Then the element $f_{p}$ in the numerical incidence algebra $I_{\#}(\widetilde{S})$ defined by 
$$
f_{p} = \delta_{p}\otimes\bigotimes_{\ell\not = p} \mu_{\widetilde{S}_{\ell}} \in I_{\#}(\widetilde{S}_{p})\otimes\bigotimes_{\ell\not = p}\!' I_{\#}(\widetilde{S}_{\ell}) \cong I_{\#}(\widetilde{S})
$$
is a relative zeta function: $\zeta_{\widetilde{S}}*f_{p} = (j_{p})_{*}\zeta_{\widetilde{S}_{p}}$, where $j_{p} : \widetilde{S}_{p}\hookrightarrow\widetilde{S}$ is the inclusion as a submonoid. That is, $f_{p}$ cuts out the submonoid $\widetilde{S}_{p}$. Passing to Dirichlet series, this becomes a more familiar statement: 
$$
\zeta_{\Q}(s)\prod_{\ell\not = p} (1 - \ell^{-s}) = \frac{1}{1 - p^{-s}}. 
$$
That is, inverse prime factors are relative zeta functions for the complementary prime factors. This can be done for any number of prime factors. 
\end{ex}

\begin{ex}
\label{ex:relquad}
Let $K/\Q$ be a quadratic number field, $\widetilde{S} = \N^{\times}$ and $\widetilde{T} = I_{K}^{\times}$. By Theorem~\ref{thm:mainthmred}, 
$$
\widetilde{N}_{*}\zeta_{\widetilde{T}} + \zeta_{\widetilde{S}}*\widetilde{L}(\chi)^{-} \cong \zeta_{\widetilde{S}}*\widetilde{L}(\chi)^{+}
$$
so $\widetilde{L}(\chi)^{+}$ and $\widetilde{L}(\chi)^{-}$ constitute a ``decomposed'' relative zeta function, similar to $\mu_{0}$ and $\mu_{1}$ in Example~\ref{ex:relMob}. As discussed in Remark~\ref{rem:redcard}, passing to the numerical incidence algebra realizes the quadratic character $\chi$ as a relative zeta function:
$$
\widetilde{N}_{*}\zeta_{\widetilde{T}} = \zeta_{\widetilde{S}}*(\widetilde{L}(\chi)^{+} - \widetilde{L}(\chi)^{-}) = \zeta_{\widetilde{S}}*\widetilde{L}(\chi)
$$
so $\widetilde{L}(\chi)$ cuts out the quadratic extension $\widetilde{N} : \widetilde{T}\rightarrow\widetilde{S}$. 
\end{ex}

\begin{ex}
Likewise, for $S = (\N,\mid)$ and $T = (I_{K}^{+},\mid)$, Theorem~\ref{thm:mainthm} shows that $L(\chi)^{+}$ and $L(\chi)^{-}$ determine a numerical relative zeta function: 
$$
N_{*}\zeta_{T} = \zeta_{S}*(L(\chi)^{+} - L(\chi)^{-}) = \zeta_{S}*L(\chi)
$$
and $L(\chi) = L(\chi)^{+} - L(\chi)^{-}$ again cuts out the quadratic extension $N : T\rightarrow S$. 
\end{ex}

\begin{ex}
On certain local factors, $\widetilde{L}_{p}(\chi)^{+}$ (resp.~$L_{p}(\chi)^{+}$) is itself a relative zeta function for $\widetilde{S}_{p}$ (resp.~for $S_{p}$). Indeed, the proof of Theorem~\ref{thm:mainthmred} (resp.~Theorem~\ref{thm:mainthm}) shows that when $p$ is a ramified or split prime, $\widetilde{N}_{*}\zeta_{\widetilde{T}_{p}} \cong \zeta_{\widetilde{S}_{p}}*\widetilde{L}_{p}(\chi)^{+}$ (resp.~$N_{*}\zeta_{T_{p}} \cong \zeta_{S_{p}}*L_{p}(\chi)^{+}$), so we can think of $\widetilde{L}_{p}(\chi)^{+}$ as cutting out the quadratic extension $\widetilde{N} : \widetilde{T}_{p}\rightarrow\widetilde{S}_{p}$ (resp.~$L_{p}(\chi)^{+}$ cutting out $N : T_{p}\rightarrow S_{p}$) in those cases.

The above paragraph can be recast in geometric terms as follows. Let $X = \Spec\F_{q}$ be a scheme theoretic point and let $Y\rightarrow X$ be the finite morphism corresponding to one of the \'{e}tale algebras $\F_{q},\F_{q}\times\F_{q}$ or $\F_{q^{2}}$ over $\F_{q}$. This determines a pushforward map on effective $0$-cycles $\pi : Z_{0}^{\eff}(Y)\rightarrow Z_{0}^{\eff}(X)$ (see \cite[Ex.~3.12]{kob} for more details) which further induces a map of reduced incidence algebras taking $\zeta_{Y}\in\widetilde{I}(Z_{0}^{\eff}(Y))$ to one of the following functors in $\widetilde{I}(Z_{0}^{\eff}(X))$: 
$$
\pi_{*}\zeta_{Y} = \begin{cases}
    \zeta_{X}, &\text{if } Y = \Spec\F_{q}\\
    \zeta_{X}*\zeta_{X}, &\text{if } Y = \Spec(\F_{q}\times\F_{q})
\end{cases}
$$
In both of these cases, $\pi$ is cut out by a relative zeta function: $\delta$ when $Y = \Spec\F_{q}$ and a copy of $\zeta_{X}$ when $Y = \Spec(\F_{q}\times\F_{q})$. However, when $Y = \Spec\F_{q^{2}}$, $\pi_{*}$ only has a relative zeta in the numerical incidence algebra, namely $\widetilde{L}_{p}(\chi)^{+} - \widetilde{L}_{p}(\chi)^{-}$ - explicitly, this is the inert case in the proof of Theorem~\ref{thm:mainthmredloc}. 
\end{ex}

%% file: future.tex
There are two main lines of inquiry which we intend to investigate further in future work.\\

{\bf Other Quadratic Extensions.} Consider a hyperelliptic curve $C/\F_{q}$, i.e.~an algebraic curve which admits a finite map $C\rightarrow\P^{1}$ of degree $2$. Such a map determines a poset map on effective $0$-cycles $\pi : Z_{0}^{\eff}(C)\rightarrow Z_{0}^{\eff}(\P^{1})$. Viewing $C\rightarrow\P^{1}$ as a geometric analogue of a quadratic number field $K/\Q$, we should expect a formula for the zeta function of $C$ in terms of the zeta function of $\P^{1}$. 

\begin{question}
Is there a formula 
$$
\pi_{*}\zeta_{C} \cong \zeta_{\P^{1}}*L(C)
$$
in any of the following algebras: the reduced incidence algebra $\widetilde{I}(Z_{0}^{\eff}(\P^{1}))$, the full incidence algebra $I(Z_{0}^{\eff}(\P^{1}))$, the reduced numerical incidence algebra $\widetilde{I}_{\#}(Z_{0}^{\eff}(\P^{1}))$ or the full numerical incidence algebra $I_{\#}(Z_{0}^{\eff}(\P^{1}))$? 
\end{question}

Here, $L(C)$ plays the role of a quadratic character in this geometric setting and should be equal to or closely related to the $L$-function of the curve $C$, depending on the ramification behavior of $\pi$ at the points over $\infty$. In the language of Section~\ref{sec:relzeta}, this would mean $L(C)$ is a relative zeta function for the map $\pi$. More generally, for any degree $2$ ramified cover of algebraic curves $\pi : C\rightarrow D$ over $\F_{q}$, there should be a relative zeta function $L(C/D)$ such that $\pi_{*}\zeta_{C} \cong \zeta_{D}*L(C/D)$. We are currently adapting our techniques in the present article to tackle this question and plan to share our findings in a sequel paper. 

Beyond zeta functions, quadratic extensions of other objects carry important topological information. Consider the case of a degree $2$ branched cover of Riemann surfaces $\pi : Y\rightarrow X$. The Riemann--Hurwitz formula relates the Euler characteristics of $X$ and $Y$: 
$$
\chi(Y) = 2\chi(X) - \sum_{y \text{ ram.}} (e_{y} - 1)
$$
where the sum is over all ramification points $y\in Y$ and the number $e_{y}$ is the ramification index of $\pi$ at $y$. This can be encoded using generating functions as the formula 
\begin{equation}\label{eq:RH}
(1 - t)^{-\chi(Y)} = (1 - t)^{-2\chi(X)}\prod_{y \text{ ram.}} (1 - t)^{e_{y} - 1}. 
\end{equation}

\begin{question}
Does the formula (\ref{eq:RH}) lift to an objective formula in some incidence algebra attached to $X$? 
\end{question}

A likely candidate is $I(SX)$ where $SX$ is the simplicial complex of $X$, since it is a simplicial object with a well-known relation to the Euler characteristic. We plan to revisit this example in future work.\\

{\bf Higher Degree Extensions.} Formula (\ref{eq:quadzeta}) is a special case of the factorization of the zeta function of an abelian number field into Dirichlet $L$-functions (or more generally, the zeta function of a Galois extension of number fields into Artin $L$-functions). For example, when $K/\Q$ is a cyclic extension of degree $n$, $\zeta_{K}(s)$ factors as 
\begin{equation}\label{eq:cyczeta}
\zeta_{K}(s) = \zeta_{\Q}(s)\prod_{j = 1}^{n - 1} L(\chi_{j},s)
\end{equation}
where $\chi_{j}$ are the nontrivial irreducible characters of the Galois group $G = \Gal(K/\Q) \cong \Z/n\Z$. As arithmetic functions, these characters take values in the set $\mu_{n}\subseteq\C$ of roots of unity (or $0$), so they cannot be categorified (in the sense of objective linear algebra) using sets alone. Instead, we propose to categorify such higher order characters using $G$-representations and promote formula (\ref{eq:cyczeta}) to an objective formula in the incidence algebra of $\N^{\times}$ in the category of simplicial $G$-representations. 

\begin{question}
Does formula (\ref{eq:cyczeta}) lift to an objective formula in the incidence algebra $I(\N^{\times})$, constructed in the category of simplicial $G$-representations? 
\end{question}

Work on this question will be carried out in a future article.